\documentclass[11pt, reqno, oneside]{amsart}
\usepackage[margin=1in]{geometry}

\usepackage[
    backend=biber,
    style=alphabetic,
    sorting=anyt,
    maxalphanames=99,
    maxbibnames=99,
    backref=true
]{biblatex}
\DeclareFieldFormat[article]{pages}{#1}
\renewbibmacro{in:}{}
\AtEveryBibitem{\clearfield{month}}
\AtEveryBibitem{\clearfield{day}}
\AtBeginBibliography{\footnotesize}

\DefineBibliographyStrings{english}{
    backrefpage = {cited on page},
    backrefpages = {cited on pages},
}

\usepackage{amsmath, amsthm, mathtools, amssymb}
\usepackage{enumitem}
\usepackage{stmaryrd} 
\SetSymbolFont{stmry}{bold}{U}{stmry}{m}{n}

\usepackage[x11names]{xcolor}

\usepackage{orcidlink}

\usepackage{subcaption}

\usepackage{tikz}
\usetikzlibrary{positioning}
\usetikzlibrary{fit}
\usetikzlibrary{external}

\tikzset{every node/.style={font=\normalsize}}
\tikzset{corners/.style={draw,fit={#1},rectangle,inner sep=0}}

\pgfdeclarelayer{dots}
\pgfsetlayers{main,dots}

\usepackage{tabto}

\usepackage{hyperref}
\hypersetup{
    colorlinks=true,
    linkcolor=blue,
    citecolor=SpringGreen3
}

\usepackage{mathrsfs}
\newcommand{\mf}{\mathfrak}
\newcommand{\mc}{\mathcal}

\newcommand{\msf}{\mathsf}

\newcommand{\mbf}{\mathbf}

\usepackage{bm}

\newcommand{\R}{\mathbb{R}}  
\newcommand{\Z}{\mathbb{Z}}
\newcommand{\Zpos}{\mathbb{Z}_{\ge0}}

\newcommand{\Q}{\mathbb{Q}}

\newcommand{\e}{\varepsilon}
\renewcommand{\d}{\delta}
\renewcommand{\o}{\omega}
\newcommand{\bE}{\mathbb{E}}

\newcommand{\norm}[1]{\left\lVert#1\right\rVert}

\DeclarePairedDelimiter\floor{\lfloor}{\rfloor}
\newcommand{\lb}{\llbracket}
\newcommand{\rb}{\rrbracket}
\newcommand{\ls}{\lesssim}
\newcommand{\gs}{\gtrsim}
\newcommand{\wh}{\widehat}
\newcommand{\wt}{\widetilde}
\DeclareMathOperator*{\argmax}{argmax}

\newcommand{\polylog}{\mathrm{polylog}}

\makeatletter
\newcommand{\E}[1][\@nil]{%
    \def\tmp{#1}%
    \ifx\tmp\@nnil
    \mathbb{E}
    \else
    \mathbb{E}\left[#1\right]
    \fi}
\renewcommand{\P}[1][\@nil]{%
    \def\tmp{#1}%
    \ifx\tmp\@nnil
    \mathbb{P}
    \else
    \mathbb{P}\left(#1\right)
    \fi}

\DeclareMathOperator{\Var}{Var}
\newcommand{\law}{\overset{d}{=}}
\newcommand{\dto}{\xrightarrow{\;d\;}}
\newcommand{\pto}{\xrightarrow{\;p\;}}
\newcommand{\1}{\mathbf{1}}
\newcommand\given[1][]{\:\vert\:#1}

\renewcommand{\H}{\mathcal{H}}
\newcommand{\V}{\mathcal{V}}
\newcommand{\BB}{\mathcal{B}}

\newcommand{\G}{\Gamma}
\newcommand{\g}{\gamma}
\newcommand{\gap}{\Upsilon}

\newcommand{\A}{\mathsf{A}}
\newcommand{\B}{\mathsf{B}}
\newcommand{\EE}{\mathsf{E}}
\newcommand{\F}{\mathsf{F}}
\newcommand{\FF}{\mathscr{F}}
\newcommand{\GG}{\mathscr{G}}

\newcommand{\QQ}{\mathbf{Q}}
\newcommand{\cons}{\mathrm{con}}
\newcommand{\hwy}{\mathrm{hwy}}
\newcommand{\exc}{\mathrm{exc}}
\newcommand{\Exc}{\Pi^{\exc}}

\newcommand{\bkwt}{\mathsf{X}}
\newcommand{\vtwt}{\mathsf{Y}}
\newcommand{\bk}{\mathrm{bulk}}
\newcommand{\f}{\mathrm{full}}
\newcommand{\ee}{\mathbf{e}_2}

\usepackage[capitalize]{cleveref}

\crefname{figure}{Figure}{Figures}

\theoremstyle{plain}
\newtheorem{theorem}{Theorem}[section]
\newtheorem{lemma}[theorem]{Lemma}

\newtheorem{corollary}[theorem]{Corollary}
\newtheorem{definition}[theorem]{Definition}

\theoremstyle{definition}
\newtheorem{remark}[theorem]{Remark}

\numberwithin{equation}{section}

\title[Pinning, diffusive fluctuations, and Gaussian limits for half-space polymers]{Pinning, diffusive fluctuations, and Gaussian limits\\ for half-space directed polymer models}

\usepackage[foot]{amsaddr}

\author{Victor Ginsburg \orcidlink{0000-0001-9399-6748}}
\address{Department of Mathematics, UC Berkeley, Berkeley, CA, USA}
\email{victor@math.berkeley.edu}

\begin{document}

\begin{abstract}
    Half-space directed polymers in random environments are models of interface growth in the presence of an attractive hard wall. They arise naturally in the study of wetting and entropic repulsion phenomena. \cite{KarDepinningQuenchedRandomness1985} predicted a ``depinning" phase transition as the attractive force of the wall is weakened.
    This phase transition has been rigorously established for integrable models of half-space last passage percolation, i.e. half-space directed polymers at zero temperature, in a line of study tracing back to the works \cite{BRAlgebraicAspectsIncreasing2001,BRAsymptoticsMonotoneSubsequences2001,BRSymmetrizedRandomPermutations2001}.
    On the other hand, for integrable positive temperature models, the first rigorous proof of this phase transition has only been obtained very recently through a series of works \cite{BWIdentityDistributionFullSpace2023,IMSSolvableModelsKPZ2022, BCDKPZExponentsHalfspace2023,DZHalfspaceLoggammaPolymer2024} on the half-space log-Gamma polymer.
    In this paper we study a broad class of half-space directed polymer models with minimal assumptions on the random environment.
    We prove that an attractive force on the wall strong enough to macroscopically increase the free energy induces phenomena characteristic of the subcritical ``bound phase," namely the pinning of the polymer to the wall and the diffusive fluctuations and limiting Gaussianity of the free energy.
    Our arguments are geometric in nature and allow us to analyze the positive temperature and zero temperature models simultaneously.
    Moreover, given the macroscopic free energy increase proved in \cite{IMSSolvableModelsKPZ2022} for the half-space log-Gamma polymer, our arguments can be used to reprove the results of \cite{IMSSolvableModelsKPZ2022,DZHalfspaceLoggammaPolymer2024} on polymer geometry and free energy fluctuations in the bound phase.
\end{abstract}

\tikzexternaldisable 
\maketitle
\tikzexternalenable

\setcounter{tocdepth}{1}

\tableofcontents


\section{Introduction, main results, and proof ideas}

Directed polymers in random environments, introduced in 
\cite{HHPinningRougheningDomain1985, ISDiffusionDirectedPolymers1988,  BolNoteDiffusionDirected1989},
are a well-studied family of models in mathematical physics.
The full-space directed polymer is widely believed to belong to the KPZ universality class.
We refer the reader to the books \cite{GiaRandomPolymerModels2007, denRandomPolymersEcole2009, ComDirectedPolymersRandom2017} for further background on directed polymers.

Half-space directed polymers in random environments were introduced by Kardar \cite{KarDepinningQuenchedRandomness1985} as a natural model for wetting and entropic repulsion phenomena that occur as an interface grows in the presence of an attractive hard wall \cite{AbrSolvableModelRoughening1980, BHLCriticalWettingThree1983, PSWSystematicsMultilayerAdsorption1982}.
Kardar predicted a ``depinning" phase transition as the attractive force of the wall is weakened: in the subcritical or ``bound'' phase, the polymer is ``pinned" to the wall; in the supercritical or ``unbound" phase, the polymer is entropically repulsed away from the wall.

This phase transition was first rigorously established for geometric and Poissonian half-space last passage percolation (LPP), two integrable zero temperature half-space directed polymer models, by Baik--Rains \cite{BRAlgebraicAspectsIncreasing2001,BRAsymptoticsMonotoneSubsequences2001,BRSymmetrizedRandomPermutations2001}.
They proved that the last passage time (i.e. zero temperature free energy) exhibits Gaussian statistics in the bound phase and KPZ universality class statistics in the critical and unbound phases.
Analogous results were later obtained for exponential half-space LPP by Sasamoto--Imamura \cite{SIFluctuationsOneDimensionalPolynuclear2004} and Baik--Barraquand--Corwin--Suidan \cite{BBCSFacilitatedExclusionProcess2018,BBCSPfaffianSchurProcesses2018}.

A recent flurry of activity has led to a comparable mathematical understanding of the depinning phase transition for integrable positive temperature half-space polymer models.
We only mention a handful of works in the following paragraph (also in the preceding paragraph), and we encourage the reader to consult \cite[Section 1.4]{BCDKPZExponentsHalfspace2023} for a far more comprehensive review of the literature on this phase transition in integrable half-space models.

The depinning phase transition for the point-to-line half-space log-Gamma (HSLG) polymer has recently been proved by Barraquand--Wang \cite{BWIdentityDistributionFullSpace2023}, and for the point-to-point HSLG polymer by Imamura--Mucciconi--Sasamoto \cite{IMSSolvableModelsKPZ2022}.
Very recently Barraquand--Corwin--Das \cite{BCDKPZExponentsHalfspace2023} extended the results of \cite{IMSSolvableModelsKPZ2022} on the HSLG polymer in the unbound phase, and moreover established the KPZ exponents ($1/3$ for the free energy fluctuations, $2/3$ for the transversal correlation length) in the critical and supercritical regimes.
The main technical contribution of \cite{BCDKPZExponentsHalfspace2023} is the construction of the HSLG line ensemble, a Gibbsian ensemble of half-infinite lines whose top line is the point-to-point HSLG polymer free energy.
In \cite{DZHalfspaceLoggammaPolymer2024}, Das--Zhu applied the Gibbs property (invariance under local resampling) of the HSLG line ensemble to confirm the predicted pinning of the HSLG polymer to the wall in the bound phase.
Specifically, \cite{DZHalfspaceLoggammaPolymer2024} proved that the endpoint of the point-to-line HSLG polymer typically lies within an $O(1)$-neighborhood of the wall.
Gibbsian line ensembles have been a focal point in the study of random planar growth models since their introduction in the seminal work \cite{CHBrownianGibbsProperty2014} of Corwin--Hammond, but a half-space Gibbsian line ensemble has yet to be constructed for other integrable polymer models, leaving open in those settings an analysis analogous to \cite{DZHalfspaceLoggammaPolymer2024}.

All the works mentioned so far depend essentially on exact formulas and combinatorial identities available for the integrable models studied therein.
It is expected that such methods cannot be adapted to non-integrable settings.
However, the depinning phase transition is predicted for a quite broad class of half-space polymer models.
This prediction motivates the present paper: our main contribution is to establish a robust criterion for the bound phase that applies to many non-integrable models.
Before describing our results, let us conclude this discussion by mentioning a related line of work.

Recently there has been an effort to develop geometric techniques that are applicable to broad classes of polymer models, and with which sharp results can be obtained given mild inputs from integrable probability.
This originated with the pioneering work of Basu--Sidoravicius--Sly \cite{BSSLastPassagePercolation2016}, where they studied a full-space LPP model through the geometry of its geodesics (i.e. polymers at zero temperature).
Their work is in fact closely related to the present paper, and we will discuss it further in \cref{rem:alternate_proof_via_bigeodesic}.
Following \cite{BSSLastPassagePercolation2016}, a number of works have used polymer geometry as a means to probe the mechanisms underlying KPZ universality phenomena.
One such work is \cite{GHOptimalTailExponents2023}, which shares the present paper's theme of obtaining sharp fluctuation estimates without integrable inputs.
In \cite{GHOptimalTailExponents2023} the authors studied full-space LPP models satisfying two natural hypotheses: concavity of the limit shape and stretched exponential concentration of the last passage time.
They used a geometric argument to upgrade these hypotheses to the optimal tail exponents for the last passage time.
Their results and techniques give a geometric explanation for these optimal tail exponents, which previously had been predicted only on the basis of their appearance in integrable LPP models through correspondences with random matrix theory (the optimal exponents, $3$ for the lower tail and $3/2$ for the upper tail, match those of the Tracy--Widom GUE distribution).

We turn now towards defining the half-space directed polymer model and formulating our main results.
The reader may find it helpful to glance at \cref{fig:outline} (below) while parsing the coming definitions.

\subsection{Model}\label{sec:model}
Let $\H\coloneqq \{(x,t)\in\Z^2:x\ge 0\}$ be the half-space bounded on the left by the vertical axis $\V\coloneqq \{(0,t):t\in\Z\}$.
For $-\infty\le t_1\le t_2\le \infty$ we write $\lb t_1,t_2\rb\coloneqq [t_1,t_2]\cap\Z$, and define
\[
    \H_{\lb t_1,t_2\rb}\coloneqq \Zpos\times\lb t_1,t_2\rb\quad\text{and}\quad
    \V_{\lb t_1,t_2\rb}\coloneqq \{0\}\times\lb t_1,t_2\rb,
\]
where $\Zpos$ denotes the set of nonnegative integers.

We fix random variables $\bkwt,\vtwt$ with the following properties:
\footnote{We use the first two properties \ref{P1}, \ref{P2} throughout the paper, while \ref{P3} is imposed for technical reasons and plays no role in most of our arguments.
We will discuss these assumptions further in \cref{rem:optimal-exponents,rem:extend_to_hslg,rem:unbounded-support}.}
\begin{enumerate}[label=(\alph*)]
    \item\label{P1} $\bkwt>0$ and $\vtwt>0$ almost surely.
    \item\label{P2} $\bkwt$ and $\vtwt$ are subexponential, i.e. there exists $z>0$ such that $\bE[e^{z\bkwt}]<\infty$ and $\bE[e^{z\vtwt}]<\infty$.
    \item\label{P3} $\bkwt$ and $\vtwt$ have unbounded supports, i.e. $\P(\bkwt>z)>0$ and $\P(\vtwt>z)>0$ for all $z\in\R$.
\end{enumerate}
We also fix a collection of independent random variables $\o=(\o(x,t))_{(x,t)\in\H}$ indexed by $\H$ such that
\[
    \o(x,t)\law \begin{cases}
        \bkwt,&\quad x>0,\\
        \vtwt,&\quad x=0.
    \end{cases}
\]
We refer to $\o$ as the \emph{environment}, and to the individual variables $\o(x,t)$ as \emph{weights}.
We denote by $\P$ the law of $\o$ and by $\E$ the expectation with respect to $\P$.

Fix $(x_1,t_1),(x_2,t_2)\in\H$ with $t_1<t_2$.
By a \emph{path $\pi$ from $(x_1,t_1)$ to $(x_2,t_2)$} we mean a collection of points
\[
    \pi=\{(y_1,s_1),(y_2,s_2),\dots,(y_{|\pi|},s_{|\pi|})\}
\]
with $(y_1,s_1)=(x_1,t_1)$ and $ (y_{|\pi|},s_{|\pi|}) = (x_2,t_2)$, and
\[
(y_i,s_i)\in \H,\quad\text{and}\quad (y_{i+1},s_{i+1}) - (y_i,s_i)\in\{(-1,1),(1,1)\} \quad\text{for all } i.
\]
When we want to emphasize the endpoints of a path, we will write $\pi:(x_1,t_1)\to (x_2,t_2)$.
We routinely identify paths $\pi:(x_1,t_1)\to(x_2,t_2)$ with graphs (over the vertical axis) of functions $\pi:\lb t_1,t_2\rb\to\Zpos$ with $\pi(t_1)=x_1$ and $\pi(t_2)=x_2$.
We denote by $\Pi(x_1,t_1;x_2,t_2)$ the set of all paths $\pi:(x_1,t_1)\to (x_2,t_2)$.

We define the \emph{Hamiltonian} (sometimes called \emph{energy} or \emph{weight}) of a path $\pi$ by
\[
    H(\pi)\coloneqq \sum_{(x,t)\in\pi}\o(x,t),
\] 
and the \emph{half-space directed polymer partition function} by
\[
    Z(x_1,t_1;x_2,t_2)\coloneqq 
    \sum_{\pi\in\Pi(x_1,t_1;x_2,t_2)}
    e^{H(\pi)}.
\]
The partition function is the normalizing constant for the \emph{polymer measure}, the random Gibbs measure 
on $\Pi(x_1,t_1;x_2,t_2)$ given by
\[
    \Q^{(x_1,t_1;x_2,t_2)}(\{\pi\})\coloneqq \frac{e^{H(\pi)}}{Z(x_1,t_1;x_2,t_2)}.
\]
We refer to a path $\pi$ sampled from $\Q^{(x_1,t_1;x_2,t_2)}$ as a \emph{polymer}.
We define the \emph{half-space directed polymer free energy} by
\[
    F(x_1,t_1;x_2,t_2)\coloneqq \log Z(x_1,t_1;x_2,t_2),
\]
and the \emph{half-space last passage time} by
\begin{equation*}\label{eq:lpp_time}
    L(x_1,t_1;x_2,t_2)\coloneqq \sup_{\pi\in\Pi(x_1,t_1;x_2,t_2)}
    H(\pi).
\end{equation*}
A \emph{geodesic} is a maximizer of the above supremum, i.e. a path $\G:(x_1,t_1)\to(x_2,t_2)$ satisfying
\[
H(\G)=L(x_1,t_1;x_2,t_2).
\]
The structure of the underlying lattice guarantees that for any two geodesics, the path which is pointwise the left-most of the two is also a geodesic.
This phenomenon is known as \emph{polymer ordering} in the literature (e.g. \cite[Lemma 11.2]{BSSLastPassagePercolation2016}), and we discuss it at length in \cref{sec:polymer_ordering}.
Polymer ordering implies that for any $(x_1,t_1),(x_2,t_2)\in\H$ with $\Pi(x_1,t_1;x_2,t_2)\ne\varnothing$, there exists a unique left-most geodesic $(x_1,t_1)\to(x_2,t_2)$.

\begin{center}
    \textbf{For the rest of the paper we denote by $n$ an even integer, so that $\Pi(0,0;0,n)\ne\varnothing$.}
\end{center}

\subsection{Main results}\label{sec:main_results}
We need the following definitions before formulating our results.
\begin{definition}[Bulk model]\label{def:bulk}
    Let $\o^\bk=(\o^\bk(x,t))_{(x,t)\in\H}$ be a collection of i.i.d. random variables satisfying
    \[
        \o^\bk(x,t) = \o(x,t)\; \text{ for } x>0, \quad\text{and}\quad  \o^\bk(x,t) \law \bkwt\; \text{ for } x=0.
    \]
    We note that $\o$ and $\o^\bk$ are \emph{equal} on $\H\setminus\V$, not only equal in distribution.
    We define
    \[
        F^\bk(x_1,t_1;x_2,t_2)\coloneqq \log\left(\sum_{\pi\in\Pi(x_1,t_1;x_2,t_2)}\exp\left(\sum_{(x,t)\in\pi}\o^\bk(x,t)\right)\right)
    \]
    and
    \[
        L^\bk(x_1,t_1;x_2,t_2)\coloneqq \sup_{\pi\in\Pi(x_1,t_1;x_2,t_2)}\sum_{(x,t)\in\pi}\o^\bk(x,t).
    \]
    Finally, we define
    \begin{equation}\label{eq:asymp_free_energy_bulk}
        f^\bk\coloneqq \lim_{n\to\infty}\frac{\bE[F^\bk(0,0;0,n)]}{n}
        \quad\text{and}\quad \ell^\bk\coloneqq \lim_{n\to\infty}\frac{\bE[L^\bk(0,0;0,n)]}{n}.
    \end{equation}
    The limits exist by superadditivity and are finite because $\bkwt$ is subexponential.
\end{definition}

The last passage time $L$ is the polymer free energy $F$ at zero temperature (we elaborate on this in \cref{sec:pos-zero}).
As a consequence, our arguments and results will usually apply simultaneously to $F$ and $L$.
It is therefore convenient to introduce a placeholder symbol representing either $F$ or $L$---we will use the letter $G$.
We will still refer to $G$ as the \emph{free energy}.
We denote by $g\in\{f,\ell\}$ the matching lowercase version: for example, we can rewrite \eqref{eq:asymp_free_energy_bulk} as $g^\bk\coloneqq \lim_{n\to\infty}n^{-1}\bE[G^\bk(0,0;0,n)]$.

\begin{definition}[LLN separation]\label{def:lln_separation}
    We say that the directed polymer model has \emph{law of large numbers (LLN) separation} if
    \begin{equation}\label{eq:DEF_lln_separation}
        g\coloneqq \lim_{n\to\infty}\frac{\bE[G(0,0;0,n)]}{n}>g^\bk.
    \end{equation}
    As with \eqref{eq:asymp_free_energy_bulk}, this limit exists by superadditivity, and is finite because $\bkwt,\vtwt$ are subexponential.
\end{definition}

\begin{remark}[Alternative definition of $g^\bk$]\label{rem:fullhalf}
    The bulk LLN $g^\bk$ is also the LLN for the free energy of the ``polymer excursion" in the original environment $\o$.
    More precisely, let $G^\exc(0,0;0,n)$ be the restriction of $G(0,0;0,n)$ to the set $\Exc(0,0;0,n)$ of paths $\pi\in\Pi(0,0;0,n)$ that satisfy $\pi\cap\V=\{(0,0),(0,n)\}$.
    Then
    \begin{equation}\label{eq:llnavoid}
        \lim_{n\to\infty}\frac{\bE[G^\exc(0,0;0,n)]}{n}=g^\bk.
    \end{equation}
    This identity is a corollary of the following distributional identity relating $G^\exc$ and $G^\bk$.
    Given $\pi\in\Exc(0,0;0,n)$, let $\pi'$ be the path obtained from $\pi$ by deleting the endpoints $\pi(0),\pi(n)$ and then translating by $(-1,-1)$.
    In symbols,  $\pi'(t)\coloneqq\pi(t+1)-1$ for $t\in\lb0,n-2\rb$.
    The map $\pi\mapsto\pi'$ defines a bijection $\Exc(0,0;0,n)\to \Pi(0,0;0,n-2)$, and we deduce the identity
    \begin{equation}\label{eq:llnavoid2}
        G^\exc(0,0;0,n) \law G^\bk(0,0;0,n-2)+\vtwt_1+\vtwt_2,
    \end{equation}
    where $\vtwt_1,\vtwt_2$ are independent copies of the vertical weight $\vtwt$.
    This immediately implies \eqref{eq:llnavoid}.
\end{remark}

The broad goal of this paper is to show that LLN separation \eqref{eq:DEF_lln_separation} gives rise to the bound phase: the polymer is pinned to $\V$, and the free energy has diffusive fluctuations and a Gaussian scaling limit.
The following is our main result on pinning.
\begin{theorem}[Pinning]\label{thm:pinning}
    There exist constants $C,C',C'',C'''>0$ and $k_0\ge1$ depending only on the law of $\o$, such that the following holds.
    Fix $G\in\{F,L\}$.
    Suppose the polymer model has LLN separation.
    Fix $t,x_1,x_2\ge0$ satisfying $t\ge k_0(x_1+x_2+1)+2$ and $\Pi(x_1,0;x_2,t)\ne\varnothing$.
    Also, fix $s_1,s_2\in\lb x_1+1,\,t-x_2-1\rb$ satisfying $s_2-s_1\ge k_0$.
    If $G=F$ then
    \begin{equation}\label{eq:pinning-thm1}
        \P[\Q^{(x_1,0;\,x_2,t)}\left(\pi\cap\V_{\lb s_1,\,s_2\rb} = \varnothing\right) > 
        C''e^{-C'''\,|s_2-s_1|}
        ]
        \le C\exp\left(-C'|s_2-s_1|^{1/3}\right).
    \end{equation}
    If $G=L$ and we denote by $\G$ the leftmost geodesic $(x_1,0)\to(x_2,t)$, then
    \begin{equation}\label{eq:pinning-thm2}
        \P[\G\cap\V_{\lb s_1,\,s_2\rb} = \varnothing]
        \le C\exp\left(-C'|s_2-s_1|^{1/3}\right).
    \end{equation}
\end{theorem}

\begin{remark}[Transversal fluctuations]\label{rem:reformulating_pinning}
    An immediate corollary of \cref{thm:pinning} is that LLN separation implies that the polymer has $O(1)$ transversal fluctuations.
    To see this, notice that any path $\pi:(0,0)\to(0,n)$ with $\pi(n/2)>k$ must be disjoint from the vertical segment $\V_{\lb \frac{n}{2}-k,\; \frac{n}{2}+k\rb}$.
    Therefore by \cref{thm:pinning}, LLN separation implies
    \begin{align}\label{eq:tf}
        \P[\Q^{(0,0;\,0,n)}\left(\pi(n/2) > k\right) > C''e^{-C'''k}] \le C\exp\left(-C'k^{1/3}\right)\quad\text{for all } k\ge k_0.
    \end{align}
    This is in fact much stronger than $O(1)$ transversal fluctuations: it shows that the typical quenched distribution of the polymer midpoint $\pi(n/2)$ has an exponential tail.    
    Similarly, our proof of \cref{thm:pinning} can be adapted to show that, given LLN separation, the quenched distribution of the half-space point-to-line directed polymer endpoint typically has an exponential tail.
\end{remark}

\begin{remark}[Optimal exponents]\label{rem:optimal-exponents}
    The quenched exponential tail asserted for the polymer measure in \eqref{eq:pinning-thm1} is sharp, in the sense that if  $|s_2-s_1|\to\infty$, then for any fixed $\alpha>1$ we have
    \[
        \P\left(\Q^{(0,0;\,0,n)}(\pi\cap\V_{\lb s_1,\,s_2\rb}=\varnothing)>e^{-|s_2-s_1|^{\alpha}}\right)=1-o(1).
    \]
    Indeed, our proof of \cref{thm:pinning} implies that (roughly speaking)
    \[
        \P[\Q^{(0,0;\,0,n)}\bigl(\pi\cap\V_{\lb s_1,\, s_2\rb} = \varnothing\bigr) > e^{-|s_2-s_1|^\alpha}] \approx
        \P\left(F(0,s_1;0,s_2) < F^\bk(0,s_1;0,s_2) + |s_2-s_1|^\alpha\right).
    \]
    The free energies on the right-hand side grow linearly in $|s_2-s_1|$ (provided that $f,f^\bk <\infty$, which is known to hold under very mild hypotheses on the laws of $\bkwt,\vtwt$).
    The claim follows.

    Let us also comment on the (sub-)optimality of the exponent $1/3$ in the $\P$-probability bounds of \cref{thm:pinning}.
    As the last display suggests, the $\P$-probability in \eqref{eq:pinning-thm1} is approximately
    \begin{align*}
        \P\left(F(0,s_1;0,s_2) < F^\bk(0,s_1;0,s_2) + C'''|s_2-s_1|\right).
    \end{align*}
    Choosing $C'''>0$ small enough such that $f^\bk+C'''<f$ (this is possible due to LLN separation), 
    we find that the decay rate of the above probability as $|s_2-s_1|\to\infty$ 
    is governed by the large deviations theory of the free energies.
    The large deviations theory is in turn a function of the tails of the weights comprising the environment.
    If the weights have subexponential tails---as we stipulated in \cref{sec:model}\ref{P2}---then 
    the above probability can be shown to decay as $\approx e^{-|s_2-s_1|}$ 
    (see e.g. Liu--Watbled \cite{LWExponentialInequalitiesMartingales2009}).
    In other words, the optimal exponent is $1$.
    However, above we merely asserted an exponent of $1/3$; 
    this suboptimal exponent stems from a crude large deviations estimate that we use for convenience 
    (\cref{lem:large_dev}, see also \cref{rem:suboptimal_large_dev}).
    
    The above discussion also applies to the last passage time: the optimal exponent in \eqref{eq:pinning-thm2} is $1$.

    It is worth pointing out that if the weights were not subexponential, 
    then the $\P$-probabilities in \cref{thm:pinning} would not decay exponentially fast.
    For example, suppose the weights had stretched exponential tails, 
    e.g. $\P(\o(x,t)>z)=\exp(-z^\kappa)$ for some $\kappa\in(0,1)$ and all sufficiently large $z>0$.
    Then the $\P$-probabilities in \cref{thm:pinning} could not decay faster than $\approx\exp(-|s_2-s_1|^\kappa)$,
    as the following calculation shows.

    We consider the polymer $(0,0)\to(0,n)$, and for simplicity
    we set $s_1=1, s_2=n-1$.
    \footnote{Our techniques imply an analogue of \cref{thm:pinning} for weights with stretched exponential tails. Given such a result, the calculations for the case $s_1=1,s_2=n-1$ can be adapted to treat arbitrary $s_1,s_2$. We omit the details.}
    Let $\pi_*$ be the \emph{unique} path $(0,0)\to(0,n)$ with $\pi(n/2)=n/2$.
    Since $\pi_*\cap\V_{\lb 1,n-1\rb}=\varnothing$, we have
    \begin{align*}
        \Q^{(0,0;\,0,n)}\bigl(\pi\cap\V_{\lb 1,n-1\rb}=\varnothing\bigr)
        &\ge \Q^{(0,0;\,0,n)}\bigl(\{\pi_*\}\bigr)\\
        &=\frac{1}{1+e^{-H(\pi_*)}\sum_{\pi\ne \pi_*}e^{H(\pi)}}\\
        &\ge \frac{1}{1+e^{-\o(n/2, n/2)}Z'},
    \end{align*}
    where we used that $e^{H(\pi_*)}\ge e^{\o(n/2,n/2)}$ (since the weights are all $\ge0$), and where we write
    \[
        Z'\coloneqq\sum_{\pi\ne \pi_*}e^{H(\pi)} + e^{H(\pi_*) - \o(n/2,n/2) + \xi}
    \]
    where $\xi$ is an independent copy of $\o(n/2, n/2)$.
    Since $Z'\law Z(0,0;0,n)$, it follows by the subadditive ergodic theorem that $n^{-1}\log Z' \to f$ a.s.
    Note also that $Z'$ is independent of $\o(n/2, n/2)$.
    Therefore, conditionally given 
    \[
        \o(n/2,\,n/2)>(f+1)n,
    \]
    the following holds true almost surely:
    \[
        \Q^{(0,0;\,0,n)}\bigl(\pi\cap\V_{\lb 1,n-1\rb}=\varnothing\bigr)
        \ge \frac{1}{1+e^{-(f+1)n}Z'}
        = \frac{1}{1+e^{-n+\e_n}},
    \]
    where $\e_n\coloneqq\log Z'-fn$ is a random variable with $n^{-1}\e_n\to 0$ a.s.
    In particular, as $n\to\infty$ we have
    \begin{align*}
        &\P\Bigl(\Q^{(0,0;\,0,n)}\bigl(\pi\cap\V_{\lb 1, n-1\rb}=\varnothing\bigr) > Ce^{-C'n}\Bigr)\\
        &\;\ge 
        \P\Bigl(\Q^{(0,0;\,0,n)}\bigl(\pi\cap\V_{\lb 1, n-1\rb}=\varnothing\bigr) > Ce^{-C'n}\;\Big|\;\o(n/2,n/2)>(f+1)n\Bigr)
        \P\Bigl(\o(n/2,n/2)>(f+1)n\Bigr)\\
        &\;= (1-o(1))\cdot \exp\bigl(-(f+1)^\kappa \,n^\kappa\bigr).
    \end{align*}
\end{remark}

\cref{thm:pinning} will play a central role in the proofs of our other results, beginning with the following theorem concerning the diffusive fluctuations and asymptotic Gaussianity of the free energy.
\begin{theorem}
    [Free energy statistics]
    \label{thm:main}
    Fix $G\in\{F,L\}$.
    Suppose the polymer model has LLN separation.
    Then \footnote{We adopt the usual interpretation of the asymptotic notation $\asymp$, and similarly for $\ls,\, \gs,\,\Theta(\cdot)$, etc. For definitions of these, see \cref{sec:notation}.}
    \[
        \Var(G(0,0;0,n)) \asymp n,
    \]
    where the implicit constants depend only on the law of $\o$.
    Moreover,
    \[
        \frac{G(0,0;0,n)-\bE[G(0,0;0,n)]}{\sqrt{\Var(G(0,0;0,n))}} \dto \msf{N}(0,1)\quad\text{as } n\to\infty,
    \]
    where $\msf{N}(0,1)$ denotes the standard Gaussian distribution.
\end{theorem}

\begin{remark}[Pinning suffices for \cref{thm:main}]\label{rem:pinning_suffices_for_clt}
    In our proof of \cref{thm:main}, we only use the LLN separation assumption to access the pinning estimates of \cref{thm:pinning}.
    In particular, our proof of \cref{thm:main} still goes through if the LLN separation assumption is replaced with the (a priori weaker) assumption that the polymer is pinned.
\end{remark}

\begin{remark}[Conjectural equivalence of LLN separation and bound phase phenomena]\label{rem:equivalence}
    We expect that LLN separation is in fact \emph{equivalent} to the pinning of the polymer to $\V$ and the conclusions of \cref{thm:main}.
    Indeed, suppose that the polymer is pinned.
    As indicated in \cref{rem:pinning_suffices_for_clt}, the proof of \cref{thm:main} allows to deduce that the free energy $G=G(0,0;0,n)$ converges to a Gaussian in the diffusive scaling limit.
    In particular, for any $M>0$ there exists $c=c(M)>0$ such that
    \begin{equation}\label{eq:fluc1}
        \P(G - gn \le -M\sqrt{n})
        \ge c\quad\text{for all sufficiently large } n.
    \end{equation}
    As alluded to before, \cite{BRAlgebraicAspectsIncreasing2001,BRAsymptoticsMonotoneSubsequences2001,BRSymmetrizedRandomPermutations2001,SIFluctuationsOneDimensionalPolynuclear2004, BBCSFacilitatedExclusionProcess2018,BBCSPfaffianSchurProcesses2018} proved that for integrable LPP models, $G^\bk=L^\bk$ exhibits $\Theta(n^{1/3})$ fluctuations.
    This was recently extended to positive temperature in the works \cite{IMSSolvableModelsKPZ2022,BCDKPZExponentsHalfspace2023} on the half-space log-Gamma (HSLG) polymer.    
    In particular, for these integrable models we have that
    \begin{equation}\label{eq:fluc2}
        \frac{G^\bk - g^\bk\, n}{\sqrt n} \pto 0\quad\text{as } n\to\infty.
    \end{equation}
    Recalling the identity \eqref{eq:llnavoid2}, we see that \eqref{eq:fluc1} and \eqref{eq:fluc2} together imply that $g>g^\bk$.

    The $\Theta(n^{1/3})$ fluctuations of $G^\bk$ are predicted to be universal, but a proof of this in non-integrable settings is far out of reach.
    However, for the equivalence of LLN separation and pinning, it suffices to establish \eqref{eq:fluc2}, i.e. $o(\sqrt n)$ fluctuations.
    A natural approach towards this more modest goal is to adopt the strategy pioneered by Benjamini--Kalai--Schramm \cite{BKSFirstPassagePercolation2003} in the setting of first passage percolation, where they proved sublinear variance growth by combining an innovative averaging argument with powerful hypercontractive inequalities (cf. \cref{rem:variance_for_kpz_models}).
    We leave a detailed analysis in this direction for future work.
\end{remark}

\begin{remark}[Extending to other environments]\label{rem:extend_to_hslg}
    Our arguments are robust and can be used to extend \cref{thm:pinning,thm:main} to polymer models with real-valued weights whose lower tails exhibit sufficiently rapid decay, as opposed to only positive weights as stipulated in \cref{sec:model}\ref{P1} (cf. \cref{rem:optimal-exponents}).
    This in particular includes the HSLG polymer.
    As a consequence, one can combine the LLN separation proved for the HSLG polymer in \cite{IMSSolvableModelsKPZ2022} with our methods to reprove the results of \cite{IMSSolvableModelsKPZ2022,DZHalfspaceLoggammaPolymer2024} on bound phase phenomena in the HSLG polymer.
    We did not pursue this in the present paper for simplicity: requiring the weights to be positive ensures that all free energies are positive, which simplifies a number of arguments that involve comparing free energies associated to different pairs of endpoints (see e.g. \cref{sec:pinning}).
\end{remark}

In addition to proving the pinning of the HSLG polymer, \cite[Theorem 1.7]{DZHalfspaceLoggammaPolymer2024} extended \cite{IMSSolvableModelsKPZ2022} by showing that the HSLG polymer free energy $F_{\mathrm{HSLG}}(0,0;y_n,n)$ has a Gaussian scaling limit in the bound phase for any sequence $y_n=o(\sqrt n)$.
In particular they leveraged their machinery to establish the approximation 
\begin{equation}\label{eq:HSLGcomp}
    \frac{1}{\sqrt n}\bigl|F_{\mathrm{HSLG}}(0,0;0,n) - F_{\mathrm{HSLG}}(0,0;y_n,n)\bigr| \pto  0,
\end{equation}
and then used the limiting Gaussianity of $F_{\mathrm{HSLG}}(0,0;0,n)$ proved in \cite[Theorem 6.9]{IMSSolvableModelsKPZ2022}.
It turns out that \eqref{eq:HSLGcomp} can be also established for polymer models exhibiting LLN separation via a quick application of our methods.
Our result to this effect is recorded as the following corollary, which for consistency we have formulated in the same manner as \cite[Theorem 1.7]{DZHalfspaceLoggammaPolymer2024}.
\begin{corollary}\label{cor:near_V}
    Fix $G\in\{F,L\}$ and suppose the polymer model has LLN separation.
    Fix an integer $k\ge 1$ and sequences of positive even integers $(y_{1,n})_{n\ge 0},\dots,(y_{k,n})_{n\ge 0}$ satisfying $y_{j,n}=o(\sqrt n)$ for all $j\in\lb 1,k\rb$.
    Then 
    \[
        \Var(G(0,0;y_{j,n},n))\asymp n\quad\text{for all } j\in\lb1,k\rb,
    \]
    where the implicit constants depend only on the sequences $(y_{j,n})_{n\ge0}$ and the law of $\o$.
    Moreover, if we fix a standard Gaussian random variable $\msf{z}\sim \msf{N}(0,1)$, then
    \[
        \left(
            \frac{G(0,0;y_{j,n},n)-\bE[G(0,0;y_{j,n},n)]}{\sqrt{\Var(G(0,0;y_{j,n},n))}}
        \right)_{j\in\lb 1,k\rb}
        \dto 
        (\msf{z},\dots,\msf{z})
        \quad\text{as } n\to\infty.
    \]
\end{corollary}

\begin{remark}[Unbounded support assumption]\label{rem:unbounded-support}
    Recall from \cref{sec:model}\ref{P3} that we assume the weights $\bkwt,\vtwt$ have unbounded supports.
    We only invoke this assumption while proving the lower bounds $\Var(G) \gs n$ asserted in \cref{thm:main,cor:near_V}; our proofs of all other results apply equally if $\bkwt$ and/or $\vtwt$ is bounded.
\end{remark}

\subsection{Idea of proof}\label{sec:outline}

We now outline our proofs of \cref{thm:pinning,thm:main}.
Our arguments will apply simultaneously to $F$ and $L$, but for concreteness we typically focus on $F$ throughout the paper (cf. \cref{sec:pos-zero}).

In \cref{sec:pinning} we prove \cref{thm:pinning} by combining LLN separation with a large deviations estimate for the free energy (\cref{lem:large_dev}), as alluded to in \cref{rem:optimal-exponents}.

The linear growth $\Var(F(0,0;0,n))\asymp n$ is the subject of \cref{sec:linear_variance}.
We prove the lower bound $\Var(F(0,0;0,n))\gs n$ by combining the pinning established in \cref{thm:pinning} with a general resampling-based estimate due to Newman--Piza \cite{NPDivergenceShapeFluctuations1995}.
For the upper bound $\Var(F(0,0;0,n))\ls n$, we apply the Efron--Stein inequality (recorded below as \cref{thm:efron-stein}).

\begin{remark}[Variance bounds for KPZ class models]\label{rem:variance_for_kpz_models}
To illustrate the relevance of pinning to our variance estimates, we note that \cite{NPDivergenceShapeFluctuations1995} and the Efron--Stein inequality are known to yield suboptimal variance bounds for KPZ universality class growth models.
For example, let $T_n$ be the full-space first passage time $(0,0)\to(0,n)$.
Newman--Piza \cite{NPDivergenceShapeFluctuations1995} used their framework to prove the lower bound $\Var(T_n)\gs \log n$, and Kesten \cite{KesSpeedConvergenceFirstPassage1993} showed that $\Var(T_n)\ls n$ via a martingale estimate analogous to the Efron--Stein inequality (see also the proof of \cite[Theorem 3.1]{ADH50YearsFirstpassage2017}).
These results are breakthroughs, but neither is sharp: it is predicted that $\Var(T_n)\asymp n^{2/3}$.
The fact that these methods yield sharp estimates in our setting can therefore be interpreted as a further manifestation of bound phase phenomena.
\end{remark}

We now discuss the proof of the Gaussian convergence in \cref{thm:main}.
In \cref{sec:approx_indep} we combine \cref{thm:pinning} with coalescence phenomena to establish a sort of ``decay of correlation" for the polymer.
The idea is as follows.
Fix $j,J$ satisfying
\footnote{We denote by $\polylog(n)$ an arbitrary (but fixed) function of the form $(\log n)^C$, for $C>0$ a constant.}
\[
    j=o(J),\quad j\uparrow\infty,\quad\text{and}\quad J\asymp \polylog(n).
\]
Consider the horizontal segments
\[
    \mc S_0\coloneqq \lb 0, j\rb \times\{s\},\qquad \mc T_0\coloneqq \lb 0,j\rb\times\{s+J\}.
\]
We say that the polymer $\pi\sim\Q^{(0,0;\,0,n)}$ is \emph{constrained} if it hits $\mc S_0$ and $\mc T_0$ (\cref{fig:outline}, left).
\cref{thm:pinning} implies that 
with high $\P$-probability, $\pi$ is typically constrained.
Also, we say that the polymer $\g_0\sim\Q^{(j,s;\,j,s+J)}$ is a \emph{local highway} if it begins its journey by moving quickly towards $\V$ to collect a vertical weight at height $\Theta(j+s)$, and concludes its journey by collecting a vertical weight at height $\Theta(s+J-j)$ before turning away from $\V$ towards its endpoint $(j,s+J)$ (\cref{fig:outline}, left).
By \cref{thm:pinning}, with high $\P$-probability, $\g_0$ is typically a local highway.
To explain the name ``local highway,'' we need to first describe how we use local highways to control $\pi$.

Suppose that $\pi$ is constrained and that $\g_0$ is a local highway.
It follows from the planarity of the model that $\pi$ intersects $\g_0$ a short time after passing through $\mc S_0$.
Similarly, $\pi$ also intersects $\g_0$ a short time before reaching $\mc T_0$.
Denote by $a,b$ the intersection points just described.
One can show that $\pi$ and $\g_0$ share the same conditional law (namely $\Q^{a;\,b}$)  given their respective trajectories below $a$ and above $b$.
Therefore by replacing the segment of $\pi$ lying between $a$ and $b$ with that of $\g_0$, we can couple $\pi,\g_0$ so that the two polymers coincide between $a$ and $b$ (\cref{fig:outline}, left).
This phenomenon, known in the literature as \emph{coalescence}, is discussed in detail in \cref{sec:polymer_ordering}.
The name ``local highway'' is intended to evoke the fact that the constrained polymer $\pi$ must ``merge" (coalesce) with $\g_0$ on the ``local" scale $J=o(\sqrt n)$.

\begin{figure}[t]
    \centering
    \begin{tikzpicture}
        \begin{scope}[local bounding box=scope1]
            \fill[black!15!white] (0,4.5) rectangle (2.25,5.2);

            \draw[color=gray, very thick, dotted](0,0.5)--(0.75,0.5)node[below]{$\mc S_0$};
            \draw[color=gray, very thick, dotted](0,4.5)--(0.75,4.5)node[above]{$\mc T_0$};
            \draw[color=black, ultra thick] (0,0)--(0,5.2);
        
            \draw[color={DarkOrange1}, ultra thick]
            (0.3,-0.4)
            node[below right]{$\pi$}
            to[out=95,in=-90] (0.1,0.1)
            to[out=120,in=-90] (0,0.2)
            to[out=70,in=-90] (0.1,0.35)
            to[out=70,in=-110] (0.2,0.6)
            to[out=110,in=-30] (0,0.7)
            to[out=80,in=-80] (0.2,1.1);
        
            \draw[color={DarkOrange1},ultra thick]
            (0.3,3.85)
            to[out=100,in=-60] (0.1,4.2)
            to[out=130,in=-70] (0,4.3)
            to[out=80,in=-120] (0.1,4.5)
            to[out=45,in=-90] (0.2,4.6)
            to[out=100,in=-60] (0,4.9)
            to[out=80,in=-130] (0.1,5)
            to[out=45,in=-80]
            (0.2,5.2);
        
            \draw[blue,ultra thick]
            (0.75,0.5)
            to[out=135,in=-45] (0.6, 0.8)
            to[out=100,in=-70] (0.2,1.1)
            to[out=120,in=-45] (0,1.3)
            to[out=90,in=230] ++(0.2,0.25) 
            to[out=80,in=-80] ++(-0.2,0.3) 
            to[out=90,in=-100] ++(0.1,0.15) 
            to[out=90,in=-100] ++(-0.1,0.1) 
            to[out=80,in=-100] ++(0.3,0.25) node[right]{$\g_0$}
            to[out=80,in=-100] ++(-0.1,0.15) 
            to[out=80,in=-100] ++(0.2,0.25) 
            to[out=80,in=-100] ++(-0.15,0.1)
            to[out=110,in=-80] ++(-0.05,0.1) 
            to[out=80,in=-100] ++(-0.2,0.15) 
            to[out=80,in=-100] ++(0.15,0.15)
            to[out=80,in=-100] ++(-0.15,0.2) 
            to[out=80,in=-100] ++(0.3,0.4)
            to[out=75,in=210] (0.75,4.5);
        
            \draw[blue, fill] (0.75,0.5) circle [radius = 0.07];
            \draw [blue, fill] (0.75,4.5) circle [radius = 0.07];
            \draw [color={DarkOrange1}, fill] (0.2,1.1) circle [radius = 0.07]
            node[above right]{$a$};
            \draw [color={DarkOrange1}, fill] (0.3,3.85) circle [radius = 0.07]node[below right]{$b$};
        \end{scope}
        \begin{scope}[shift={(6.5,0.1)}]
            \fill[black!15!white] (0,0.25) rectangle (2.5,2);
            \fill[black!15!white] (0,2.5) rectangle (2.5,4.25);

            \draw[color=gray, thick, |<->|](-.3,2.5)--(-.3,4.25) node[midway, left]{$K$};
    
            \draw[color=gray,thick,dotted] (0,-0.25)--(0.3,-0.25);
            \draw[color=gray,thick,dotted] (0,0.25)--(0.3,0.25);
            \draw[color=gray,thick,dotted] (0,2)--(0.3,2);
            \draw[color=gray,thick,dotted] (0,2.5)--(0.3,2.5);
            \draw[color=gray,thick,dotted] (0,4.25)--(0.3,4.25);
            \draw[color=gray,thick,dotted] (0,4.75)--(0.3,4.75);

            \draw[color=black, very thick] (0,-0.5)--(0,5) node[left]{$\V$};

            \draw[color=blue, very thick] (0.3,-0.25)
            to[out=120,in=-80] ++(-0.3,0.15)
            to[out=80,in=-100] ++(0.1,0.1)
            to[out=100,in=-80] ++(-0.1,0.1) 
            to[out=80,in=-120] (0.3,0.25)node[below right]{$\g_0$};

            \draw[color=blue, very thick] (0.3,2)
            to[out=120,in=-80] ++(-0.3,0.15)
            to[out=80,in=-100] ++(0.1,0.1)
            to[out=100,in=-80] ++(-0.1,0.1) 
            to[out=80,in=-120] (0.3,2.5)node[below right]{$\g_1$};

            \draw[color=blue, very thick] (0.3,4.25)
            to[out=120,in=-80] ++(-0.3,0.15)
            to[out=80,in=-100] ++(0.1,0.1)
            to[out=100,in=-80] ++(-0.1,0.1) 
            to[out=80,in=-120] (0.3,4.75)node[below right]{$\g_2$};

            \draw[color={DarkOrange1}, very thick] (0.2,-0.8)
            node[right]{$\pi$}
            to[out=120,in=-80] (0.1,-0.16);

            \draw[color={DarkOrange1}, very thick] (0.07,0.15)
            to[out=90,in=-90] ++(0.1,0.25) 
            to[out=80,in=-80] ++(-0.17,0.3) 
            to[out=90,in=-100] ++(0.1,0.15) 
            to[out=90,in=-100] ++(-0.1,0.1) 
            to[out=80,in=-100] ++(0.2,0.25)
            to[out=80,in=-100] ++(-0.15,0.1)
            to[out=110,in=-80] ++(-0.05,0.1) 
            to[out=80,in=-100] ++(0.15,0.15)
            to[out=80,in=-100] ++(-0.15,0.2) 
            to[out=80,in=-100] ++(0.2,0.2)
            to[out=100,in=-90] (0.1,2.05);

            \draw[color={DarkOrange1}, very thick] (0.1,2.4)
            to[out=90,in=230] ++(0.1,0.25) 
            to[out=80,in=-80] ++(-0.2,0.3) 
            to[out=90,in=-100] ++(0.1,0.15) 
            to[out=90,in=-100] ++(-0.1,0.1) 
            to[out=80,in=-100] ++(0.2,0.25)
            to[out=80,in=-100] ++(-0.2,0.1)
            to[out=80,in=-100] ++(0.15,0.15)
            to[out=80,in=-100] ++(-0.15,0.2) 
            to[out=80,in=-100] ++(0.2,0.2)
            to[out=100,in=-90] (0,4.4);

            \draw[color={DarkOrange1},very thick] (0.1,4.65) to[out=60,in=-80] (0.1,5);

            \draw [color={DarkOrange1}, fill] (0.1,-0.18) circle [radius = 0.05];
            \draw [color={DarkOrange1}, fill] (0.07,0.16) circle [radius = 0.05];
            \draw [color={DarkOrange1}, fill] (0.1,2.07) circle [radius = 0.05];
            \draw [color={DarkOrange1}, fill] (0.1,2.42) circle [radius = 0.05];
            \draw [color={DarkOrange1}, fill] (0,4.4) circle [radius = 0.05];
            \draw [color={DarkOrange1}, fill] (0.1,4.67) circle [radius = 0.05];

            \draw[blue, fill] (0.3,-0.25) circle [radius = 0.03];
            \draw[blue, fill] (0.3,0.25) circle [radius = 0.03];
            \draw[blue, fill] (0.3,2) circle [radius = 0.03];
            \draw[blue, fill] (0.3,2.5) circle [radius = 0.03];
            \draw[blue, fill] (0.3,4.25) circle [radius = 0.03];
            \draw[blue, fill] (0.3,4.75) circle [radius = 0.03];

            \node (A) at (-0.1,-0.9) {};
            \node (B) at (.75,0.25) {};
        \end{scope}

        \node [corners={([shift={(-.4,-.1)}]scope1.south west) (scope1.north east)}] (box1) {};

        \node [corners={(A) (B)}] (box2) {};

        \draw[color=black, densely dotted]
        (box1.north east)--(box2.north west);
        \draw[color=black, densely dotted]
        (box1.south east)--(box2.south west);
    \end{tikzpicture}

    \caption{\textbf{Left:} The blue path $\g_0$ is a local highway.
    The polymer $\pi$ (orange) is constrained, and hits $\g_0$ for the first time at $a$ and for the last time at $b$.
    Its trajectory between $a$ and $b$ has been replaced by that of $\g_0$.
    In accordance with our assumption $j=o(J)$, we have depicted $\mc S_0,\mc T_0$ (dotted gray line segments) as being much shorter than the vertical distance between them.
    \textbf{Right:} Three local highways $\g_0,\g_1,\g_2$ (blue) are separated by a vertical distance of $K$.
    The three corresponding polymer measures are each determined by the weights in a different $J\times J$ box (not drawn).
    As $J=o(K)$, these boxes are disjoint and therefore the three polymer measures are $\P$-independent.
    Then since $\pi$ (orange) coalesces with each $\g_i$, it follows that the laws of the segments of $\pi$ within the shaded gray horizontal strips are $\P$-independent.}
    \label{fig:outline}
\end{figure}

We now fix a parameter $K\asymp n^C$ for some constant $C\in(\frac12,1)$.
We consider vertical translates of the above construction: for $i\in\{1,2\}$, we define segments 
\[
    \mc S_i\coloneqq \mc S_0 + (0,iK),\qquad \mc T_i\coloneqq \mc T_0+(0,iK)
\]
and paths
\[
    \g_i \sim \Q^{(j,\,s+iK;\;j,\,s+iK+J)}.
\]
A union bound shows that with high $\P$-probability, the polymer $\pi$ is typically constrained (i.e. $\pi$ hits each $\mc S_i,\mc T_i$), and each $\g_i$ is typically a local highway (\cref{fig:outline}, right).
Consider boxes
\[
    \lb 0,\,J\rb\times\lb s+iK,\,s+iK+J\rb,\quad i\in\{0,1,2\}.
\]
As $J=o(K)$, these boxes are well-separated from each other.
Suppose that $\pi$ is constrained and that every $\g_i$ is a local highway.
As we have seen above, it follows that $\pi$ coalesces with each $\g_i$.
On the other hand, for $i\ne i'$, the weights inside the $i$-th box and the weights inside the $i'$-th box are independent.
Since the law of each $\g_i$ depends only on the environment $\o$ within the $i$-th box, it follows that the segment of $\pi$ between $\mc T_0$ and $\mc S_1$, and the segment of $\pi$ between $\mc T_1$ and $\mc S_2$, are independent with respect to $\P$ 
(\cref{fig:outline}, right).

In \cref{sec:approx_indep} we make this precise, and extend it to hold simultaneously for $(n/K)$-many pairs $\mc S_i,\mc T_i$ and local highways $\g_i$.
The behavior of $\pi$ inside the $i$-th box (of area $J^2\asymp\polylog(n)$) has a negligible effect on the diffusively-scaled free energy.
This leads to an approximation of the form
\[
    \frac{F(0,0;0,n)-\bE[F(0,0;0,n)]}{\sqrt n} \approx 
    \frac{1}{\sqrt{n/K}} \sum_{i=1}^{n/K}\frac{F_i-\bE[F_i]}{\sqrt K},
\]
where the $F_i$ are \emph{independent} random variables, each being (approximately) the contribution made to $F(0,0;0,n)$ by the polymer $\pi$ during its journey from $\mc T_i$ to $\mc S_{i+1}$.

Finally, since $n/K\to\infty$, we have reduced the problem to that of verifying the hypotheses of the classical Lindeberg central limit theorem for a diffusively-scaled sum of $(n/K)$-many independent random variables $X_i\coloneqq \frac{F_i-\bE[F_i]}{\sqrt K}$ with variance $\Var(X_i)\asymp 1$ (the aforementioned variance estimates imply $\Var(F_i)\asymp K$).
We show that the $X_i$ satisfy the Lindeberg condition in  \cref{sec:lindeberg} via a straightforward martingale concentration argument that does not depend on polymer pinning.

\begin{remark}[Alternative proof of Gaussian fluctuations at zero temperature]\label{rem:alternate_proof_via_bigeodesic}
    It turns out that geodesic pinning can be used to provide a much cleaner proof of \cref{thm:main} for the LPP model $G=L$ than what we outlined above.
    We sketch this now.

    Assuming LLN separation, \cref{thm:pinning} implies that the left-most geodesic $\G:(0,-m)\to(0,n)$ typically hits $\V$ at $\Theta(n+m)$-many locations.
    For $i\in\Z$, let $H_i$ be the energy that $\G$ accrues during its journey between the pair of consecutive hitting locations straddling the horizontal line $\Z\times\{i\}$.
    Using that $|\G\cap\V|=\Theta(n+m)$, one can adapt (and substantially simplify) the coalescence arguments of \cref{sec:approx_indep} to prove that the correlation between $H_i$ and $H_j$ decays as a stretched exponential in $|i-j|$.
    A further application of coalescence allows to take the limit $n,m\to\infty$, yielding a sequence $(\mf{H}_i)_{i\in\Z}$.
    One can show that $(\mf{H}_i)_{i\in\Z}$ inherits a stretched exponential rate of mixing from the prelimit, and that $(\mf{H}_i)_{i\in\Z}$ is stationary as a result of the vertical translation-invariance of the environment.
    One can further show that, under diffusive scaling, the last passage time $L(0,-m;0,n)$ is approximated by a sum of the $\mf{H}_i$.
    The asymptotic Gaussianity then follows from classical results on the central limit theorem for stationary mixing sequences (e.g. \cite{ILIndependentStationarySequences1971,BolCentralLimitTheorem1982}).
    That $\Var(L(0,-m;0,n))=\Theta(n+m)$ can be proved as in \cref{sec:linear_variance}, but with some simplifications owing to the stretched exponential mixing.

    A remark to this effect previously appeared in \cite{BSSLastPassagePercolation2016}, where the authors resolved the famous slow bond problem by establishing LLN separation for a full-space LPP model with reinforced weights on a line.
    They further observed that LLN separation implies the pinning of the geodesic to the line, and explained how this can be used to construct the process $(\mf{H}_i)_{i\in\Z}$ and deduce the Gaussian fluctuations of the last passage time.
    Subsequently, Basu--Sarkar--Sly \cite{BSSInvariantMeasuresTASEP2017} used geometric arguments to resolve several outstanding conjectures of Liggett related to the slow bond problem.
    In the course of their analysis they proved the pinning of the geodesic, and as a corollary provided the details of the argument of \cite{BSSLastPassagePercolation2016} for the Gaussian fluctuations (see \cite[Appendix B]{BSSInvariantMeasuresTASEP2017}).

    Unlike the last passage time, the positive temperature free energy depends on every path and consequently cannot be analyzed using only the elegant theory of stationary sequences.
    We therefore take a mesoscopic approach that allows to treat the zero and positive temperature models in a unified manner.
\end{remark}

\subsection{Organization of the paper}
In \cref{sec:notation} we record some notation.
The remainder of \cref{sec:preliminaries} is spent collecting general results on the directed polymer model: a large deviations estimate in \cref{sec:lln}, LLN comparisons in \cref{sec:LLN_comparisons}, the phenomena of polymer ordering and coalescence in \cref{sec:polymer_ordering}, and the correspondence between the positive temperature and zero temperature model in \cref{sec:pos-zero}.
In \cref{sec:pinning} we prove \cref{thm:pinning}.
In \cref{sec:linear_variance} we prove that the free energy $(0,0)\to(0,n)$ has variance $\Theta(n)$.
In \cref{sec:approx_indep} we construct independent random variables whose sum approximates the diffusively-scaled free energy.
In \cref{sec:near_V} we prove \cref{cor:near_V}.
In \cref{sec:lindeberg} we verify the Lindeberg condition for the random variables constructed in \cref{sec:approx_indep}, thereby completing the proof of \cref{thm:main}.

\subsection{Acknowledgements} 
I thank my advisor Shirshendu Ganguly for suggesting this problem, for numerous helpful discussions and insights, and for many comments on earlier drafts of this paper.
I also thank the two anonymous referees for their meticulous feedback that greatly improved the presentation of the paper, and for their interesting questions and suggestions, one of which led to \cref{rem:optimal-exponents}.
This work was supported by the National Science Foundation Graduate Research Fellowship Program under Grant No. DGE-2146752.


\section{Preliminaries}\label{sec:preliminaries}

\subsection{Notation}\label{sec:notation}
We denote by $C,C',C'',C'''$ deterministic, strictly positive constants whose values may change from line to line (or in the same line), and which may depend on the law of the environment $\o$, but not on any other parameters (such as $n$).

We follow the usual Landau asymptotic notation: we write $A=O(B)$ if $|A|\le CB$ for some $C>0$, and $A=\Theta(B)$ if $A=O(B)$ and $B=O(A)$.
We will frequently write $A\ls B$ instead of $A=O(B)$, $A\gs B$ instead of $B=O(A)$, and $A\asymp B$ instead of $A=\Theta(B)$.
Lastly, we write $A=o(1)$ if $\lim_{n\to\infty}A=0$, and $A=o(B)$ if $\frac{A}{B}=o(1)$.


\subsection{Large deviations}\label{sec:lln}
In this section we prove a large deviations estimate for the free energy.
For $G\in\{F,L\}$ we define
\begin{equation}\label{eq:gap}
    \gap=\gap(G)\coloneqq \frac{g-g^\bk}{5},
\end{equation}
so that LLN separation \eqref{eq:DEF_lln_separation} is equivalent to $\gap>0$.

\begin{lemma}[Free energy large deviations]\label{lem:large_dev}
    There exist constants $C,C',c_0>0$ depending only on the law of $\o$ such that the following holds.
    Fix $G\in\{F,L\}$ and suppose the polymer model has LLN separation.
    Then for all $u=(x_1,t_1), v=(x_2,t_2)\in\H$ with vertical displacement $t_2-t_1\ge c_0$ and $\Pi(u;v)\ne\varnothing$, we have that 
    \begin{equation}\label{eq:conc_vert}
        \P\left(
            \bigl|G(u;v)-\bE[G(u;v)]\bigr| > \gap |t_2-t_1|
        \right)
        \le C\exp\left(-C'\min\{1,\gap^2\} \,|t_2-t_1|^{1/3}\right)
    \end{equation}
    and that
    \begin{equation}\label{eq:conc_bulk}
        \P\left(\bigl|G^\bk(u;v)-\E\bigl[G^\bk(u;v)\bigr]\bigr| > \gap |t_2-t_1|\right)   \le C\exp\left(-C'\min\{1,\gap^2\}\,|t_2-t_1|^{1/3}\right).
    \end{equation}
\end{lemma}

We will use the following fact in the proof of \cref{lem:large_dev}.
\begin{lemma}\label{lem:LSE}
    Fix $d\ge1$ and define functions $h_1,h_2:\R^d\to\R$ by
    \[
        h_1(x)\coloneqq \log\left(\sum_{i=1}^d e^{x_i}\right),\qquad h_2(x)\coloneqq \max_{i\in\lb 1,\,d\rb} x_i. 
    \]
    Then for $k\in\{1,2\}$, the function $h_k$ satisfies
    \[
        |h_k(x)-h_k(y)| \le \max_{i\in\lb 1,\,d\rb}|x_i-y_i|,\quad\forall x,y\in[0,\infty)^d.
    \]
\end{lemma}
\begin{proof}[Proof of \cref{lem:LSE}]
    A direct calculation shows that the gradient $\nabla h_1$ has $\ell^1$-norm $\norm{\nabla h_1(x)}_1=1$ for all $x\in\R^d$. The inequality follows from the mean value theorem.
    As for $h_2$, \cref{lem:LSE} is just the reverse triangle inequality for the $\ell^\infty$-norm, restricted to the nonnegative orthant $[0,\infty)^d$.
\end{proof}

\begin{proof}[Proof of \cref{lem:large_dev}]
    For notational simplicity we only prove \eqref{eq:conc_vert}, but the same argument works for \eqref{eq:conc_bulk}. 
    By translation-invariance in the vertical direction, it suffices to treat the case $u=(x_1,0), v=(x_2,t)$.
    Consider the truncated environment $\wh\o$ given by
    \[
        \wh\o(y,s)\coloneqq \o(y,s)\wedge t^{1/3}\quad \text{for } (y,s)\in\H.
    \]
    Let $\wh{G}$ be the free energy $u\to v$ in $\wh\o$.
    For $j\in\Z$ we denote by $\FF_j$ the $\sigma$-algebra generated by $\wh\o$ up to height $j$, that is, $\FF_j\coloneqq\sigma\bigl(\wh\o(y,s):y\in\Zpos,\; s\le j \bigr)$.
    We first show that
    \begin{equation}\label{eq:conc_bounded_inc}
        \left|\E\bigl[\wh{G}\given[\FF_j]\bigr]-\E\bigl[\wh{G}\given[\FF_{j-1}]\bigr]\right| 
        \le 2t^{1/3}\quad\text{for all } j\in\lb 0,t\rb.
    \end{equation}
    For this we mimic an argument of \cite{KesSpeedConvergenceFirstPassage1993}.
    Fix a realization of the environment $\wh\o$.
    Also fix $j\in\lb0,t\rb$ and let $\wh\o'$ be the environment obtained from $\wh\o$ by replacing $\wh\o(y,j)$ with an independent copy $\wh\o'(y,j)$ for each $y\in\Zpos$.
    For any path $\pi:u\to v$, we have
    \begin{equation}\label{eq:influence_bound_kesten}
        \left|\sum_{(y,s)\in \pi}\wh\o(y,s) - \sum_{(y,s)\in\pi}\wh\o'(y,s)\right|
            = \bigl|\wh\o(\pi(j),j)-\wh\o'(\pi(j),j)\bigr| \le 2t^{1/3}.
    \end{equation}
    By \cref{lem:LSE},
    \[
        \left|\wh{G}-\wh{G}'\right|
        \le \sup_{\pi\in\Pi(u;v)}\left|\sum_{(y,s)\in \pi}\wh\o(y,s) - \sum_{(y,s)\in\pi}\wh\o'(y,s)\right|,
    \]
    where $\wh{G}'$ denotes the free energy $u\to v$ in $\wh\o'$.
    We substitute the above inequality into \eqref{eq:influence_bound_kesten} and average over $(\wh\o'(y,s))_{s\ge j,\,y\in\Zpos}$ to conclude \eqref{eq:conc_bounded_inc}.

    Having shown \eqref{eq:conc_bounded_inc}, we may apply the Azuma--Hoeffding inequality to $\wh{G}$.
    This yields absolute constants $C,C'>0$ such that for all $z>0$,
    \begin{equation}\label{eq:conc_azuma_hoeffding}
            \P\left(\bigl|\wh{G}-\bE[\wh{G}]\bigr| > z t\right)
            \le C\exp\left(-C'\frac{z^2t^2}{\sum_{i=0}^{t}t^{2/3}}\right)\le C\exp\left(-C'z^2 t^{1/3}\right).
    \end{equation}
    We now estimate the error introduced by truncating the weights. 
    We begin by observing that $\wh{G}$ and $G\coloneqq G(x_1,0;x_2,t)$ depend only on the weights inside the box
    \[
        \BB\coloneqq \lb \min\{x_1,x_2\} -t, \;\max\{x_1,x_2\}+t\rb\times \lb0,t\rb.
    \]
    As $\Pi(x_1,0;x_2,t)\ne\varnothing$, it follows that the box $\BB$ has area $|\BB|\le Ct^2$ for some absolute constant $C>0$ that does not depend on $x_1,x_2,t$.
    Therefore, a union bound over $\BB$ yields
    \begin{equation}\label{eq:conc_union_bound}
        \begin{split}
            \P\bigl(G\ne \wh{G}\bigr) 
            &\le \P[\sup_{(y,s)\in\BB}\o(y,s)>t^{1/3}]\\
            &\le Ct^2\exp(-C't^{1/3})\\
            &\le C\exp(-C't^{1/3}),
        \end{split}
    \end{equation}
    where we used the fact that $\bkwt,\vtwt$ are subexponential (\cref{sec:model}\ref{P2}).
    Combining \eqref{eq:conc_union_bound} with the inequality
    \[
        G-\wh{G} 
         \le L(x_1,0;x_2,t) + \log|\Pi(x_1,0;x_2,t)| \le \sum_{(y,s)\in\BB}\o(y,s) + t\log 4
    \]
    and the Cauchy--Schwarz inequality, we get
    \begin{equation}\label{eq:conc_expectation}
        \begin{split}
            0 \le \E\bigl[G-\wh{G}\bigr]
            &= \E[\bigl(G-\wh{G}\bigr)\1_{\{G\ne \wh{G}\}}]\\
            &\le C\,\E[\left(\sum_{(y,s)\in\BB}\o(y,s)+t\log 4\right)^2]^{1/2}\exp(-C't^{1/3})\\
            &\le Ct^2\exp(-C't^{1/3})\\
            &\le C\exp(-C't^{1/3}),
        \end{split}
    \end{equation}
    where $C,C'$ depend only on the law of $\o$.
    By our hypothesized LLN separation (i.e. $\gap>0$), we can choose a constant $c_0>0$ depending only on $C,C',\gap$, such that
    \begin{equation}\label{eq:purpose_of_sep}
        C\exp(-C't^{1/3}) <\frac{\gap}{2}t\quad\text{for all } t\ge c_0.
    \end{equation}
    Then, combining \eqref{eq:conc_azuma_hoeffding}, \eqref{eq:conc_union_bound}, and \eqref{eq:conc_expectation}, we conclude that for $t\ge c_0$,
    \begin{equation}\label{eq:conc_final}
        \begin{split}
            \P\Bigl(\bigl|G-\bE[G]\bigr|>\gap t\Bigr)
            &\le \P\left(\bigl|\wh G-\bE[\wh G]\bigr| + \bigl|\bE[G]-\bE[\wh G]\bigr|>\gap t\right) + \P\Bigl(G\ne \wh G\Bigr)\\
            &\le \P\left(\bigl|\wh G-\bE[\wh G]\bigr|>\frac{\gap}{2}t\right) + C\exp\left(-C't^{1/3}\right)\\
            &\le C\exp\left(-C'\gap^2t^{1/3}\right)+ C\exp\left(-C't^{1/3}\right).
        \end{split}
    \end{equation}
    This proves \cref{lem:large_dev}.
\end{proof}

\begin{remark}[Suboptimality of \cref{lem:large_dev}]\label{rem:suboptimal_large_dev}
    \cref{lem:large_dev} is far from sharp.
    For instance, the proof implies the same result with $\gap$ replaced by any fixed $z>0$, provided that $c_0$ is increased (depending on $z$).
    Faster tail decay rates are also known
    (e.g. \cite{LWExponentialInequalitiesMartingales2009}, see also \cref{rem:optimal-exponents}).
    However, we are not aware of a suitable estimate in the literature that applies simultaneously to the zero temperature and positive temperature models.
    \cref{lem:large_dev} suffices for our purposes as-is, so we did not attempt to optimize it further.
\end{remark}


\subsection{LLN comparisons}\label{sec:LLN_comparisons}
We record two lemmas comparing LLNs of various free energies, with the aim of streamlining our upcoming applications of LLN separation.

We introduce full-space analogues of the objects from \cref{def:bulk}.
Let $\Pi_\f(x_1,t_1;x_2,t_2)$ be the set of \emph{all} lattice paths with steps in $\{(-1,1),(1,1)\}$ joining $(x_1,t_1)$ to $(x_2,t_2)$, not only those confined to the half-space $\H$.
We extend the environment $\o^\bk$ to a full-space environment of i.i.d. copies of $\bkwt$ and define
\[
    F^\bk_\f(x_1,t_1;x_2,t_2)\coloneqq \log\left(\sum_{\pi\in\Pi_\f(x_1,t_1;x_2,t_2)}\exp\left(\sum_{(x,t)\in\pi}\o^\bk(x,t)\right)\right)
\]
and
\[
    L^\bk_\f(x_1,t_1;x_2,t_2)\coloneqq \sup_{\pi\in\Pi_\f(x_1,t_1;x_2,t_2)}\sum_{(x,t)\in\pi}\o^\bk(x,t).
\]
We fix $G\in\{F,L\}$ and write $g\coloneqq f$ or $g\coloneqq \ell$ accordingly.
Consider the cones
\[
    D\coloneqq\bigl\{\theta\in\R^2\setminus\{(0,0)\} : \theta_2\ge \theta_1\ge0\bigr\}, \qquad D_\f\coloneqq\bigl\{\theta\in\R^2\setminus\{(0,0)\} : \theta_2\ge |\theta_1|\bigr\}.
\]
We extend \eqref{eq:asymp_free_energy_bulk} by defining
\[
    g^\bk(\theta)\coloneqq \lim_{n\to\infty}\frac{\E[G^\bk\bigl(0,0;2\floor{n\theta_1/2}, 2\floor{n\theta_2/2}\bigr)]}{n},\qquad \theta\in D
\]
and
\[
    g^\bk_\f(\theta)\coloneqq \lim_{n\to\infty}\frac{\E[G^\bk_\f\bigl(0,0;2\floor{n\theta_1/2}, 2\floor{n\theta_2/2}\bigr)]}{n},\qquad \theta\in D_\f,
\]
where $\floor{\cdot}$ is the floor function.
These limits exist by superadditivity.
It follows that $g^\bk(\lambda\theta)=\lambda g^\bk(\theta)$ and $g^\bk_\f(\lambda\theta)=\lambda g^\bk(\theta)$ for any $\lambda>0$.

Write $\ee\coloneqq (0,1)\in D$.
The vertical LLN $g^\bk$ appearing in \eqref{eq:asymp_free_energy_bulk} is presently denoted by $g^\bk(\ee)$.
The following lemma asserts that in an i.i.d. environment, the half-space vertical LLN coincides with the full-space vertical LLN.
\begin{lemma}\label{lem:iid_full=half}
    $g^\bk_\f(\ee) = g^\bk(\ee)$.
\end{lemma}
\begin{proof}
    The inequality $g^\bk_\f(\ee) \ge g^\bk(\ee)$ follows from the fact that $\Pi(0,0;0,n)\subset \Pi_\f(0,0;0,n)$.
    For the reverse inequality, we first fix $\e>0$ and choose an even integer $k=k(\e)>0$ such that
    \begin{equation}\label{eq:lln_scale_k}
        \frac{\E[G^\bk_\f(0,0;0,k)]}{k} \ge g^\bk_\f(\ee)-\e.
    \end{equation}
    By superadditivity and the fact that the free energies are $\ge 0$ almost surely, we have
    \begin{equation}\label{eq:lln_bulkdecomp}
        G^\bk(0,0;0,n) \ge \sum_{i=2}^{n/k - 3} G^\bk\left(\frac{k+4}{2},\; ik;\; \frac{k+4}{2},\;(i+1)k\right).
    \end{equation}
    On the other hand, any full-space path $\pi\in\Pi_\f(0,0;0,n)$ satisfying
    \begin{equation}\label{eq:lln_bulkdecomp2}
        \pi(ik) =\frac{k+4}{2} \quad\text{for all } i\in\lb 2,\,n/k - 2\rb
    \end{equation}
    must also satisfy $\pi\cap \H_{\lb 2k,\, n-2k\rb} \subset \H\setminus\V$ (cf. \cref{rem:reformulating_pinning}).
    Therefore for $i\in\lb 2,\,n/k-3\rb$,
    \begin{equation}\label{eq:sep_from_vert_iid}
        G^\bk\left(\frac{k+4}{2},\; ik;\; \frac{k+4}{2},\; (i+1)k\right) = G^\bk_\f\left(\frac{k+4}{2},\; ik;\; \frac{k+4}{2},\; (i+1)k\right).
    \end{equation}
    By \eqref{eq:lln_scale_k}, \eqref{eq:lln_bulkdecomp}, \eqref{eq:lln_bulkdecomp2}, \eqref{eq:sep_from_vert_iid}, and vertical translation-invariance,
    \[
        \begin{split}
            \frac{\E[G^\bk(0,0;0,n)]}{n} &\ge \frac{1}{n}\sum_{i=2}^{n/k-3} \E[G^\bk_\f\left(\frac{k+4}{2},\; ik;\; \frac{k+4}{2},\; (i+1)k\right)]\\
            &= \left(\frac{1}{k}-\frac{4}{n}\right)\E[G^\bk_\f\left(0,0;0,k\right)]\\
            &\ge g^\bk_\f(\ee) - \e - O(1/n).
        \end{split}
    \]
    Let $n\to\infty$ to get $g^\bk(\ee)\ge g^\bk_\f(\ee)-\e$, then let $\e\downarrow0$.
\end{proof}

The next lemma will help us establish pinning for polymers whose endpoints do not lie on $\V$.
\begin{lemma}[Vertical LLN is largest]\label{lem:vertical_maximizes_g}
    For all $\theta\in D_\f$ with $\ell^1$-norm $\norm{\theta}_1=1$,
    \[
        g^\bk(\ee) \ge g^\bk_\f(\theta).
    \]
    In particular, LLN separation \eqref{eq:DEF_lln_separation} implies that $g>g^\bk_\f(\theta)$ for all $\theta\in D_\f$ with $\norm{\theta}_1=1$.
\end{lemma}
\begin{proof}
    Write $\mbf 0 \coloneqq (0,0)$.
    By superadditivity,
    \[
        \E[G^\bk_\f(\mbf{0};n\ee)] \ge \E[G^\bk_\f(\mbf{0}; 2\floor{n\theta/4})]
        + \E[G^\bk_\f(2\floor{n\theta/4}; n\ee)] - O(1),
    \]
    where the $O(1)$ error is from double-counting $\o^\bk(2\floor{n\theta/4})$.
    Here we apply the floor function $\floor{\cdot}$ entry-wise. 
    On the other hand, by symmetry and the fact that $\norm{\theta}_1=1$, we have
    \[
        \E[G^\bk_\f(\mbf{0};2\floor{n\theta/4})]
        = \E[G^\bk_\f(2\floor{n\theta/4}; n\ee)]
        +O(1),
    \]
    where the $O(1)$ error is from taking $\floor{\cdot}$.
    Therefore
    \[
        g^\bk_\f(\ee) = g^\bk(\ee) \ge 2g^\bk_\f(\theta/2)=g^\bk_\f(\theta),
    \]
    where the first equality is by \cref{lem:iid_full=half}.
\end{proof}

In the sequel we resume our use of the abbreviation $g^\bk\coloneqq g^\bk(\ee)$.


\subsection{Polymer ordering and coalescence}\label{sec:polymer_ordering}
Recall the polymer ordering phenomenon described in \cref{sec:model}: the path which is pointwise the left-most of two geodesics is itself a geodesic.
In particular, for any $u,v\in\H$ with $\Pi(u;v)\ne\varnothing$, there exists a unique left-most geodesic $u\to v$.
This uniqueness implies that when two left-most geodesics intersect, they \emph{coalesce}, sharing as much of their remaining journeys as possible:
\begin{lemma}[Geodesic coalescence]\label{lem:coalescence}
    Fix $u,u',v,v'\in\H$ and let $\G_{u;v}:u\to v,\,\G_{u';v'}:u'\to v'$ be left-most geodesics.
    The following holds almost surely.
    If $\G_{u;v}(r)=\G_{u';v'}(r)$ for some $r$, and $\G_{u;v}(s)\ne\G_{u';v'}(s)$ for some $s>r$, then $\G_{u;v}(t)\ne\G_{u';v'}(t)$ for all $t\ge s$.
    In other words, the intersection $\G_{u;v}\cap\G_{u';v'}$ is a connected subset of $\H$.
\end{lemma}
\begin{proof}
    Suppose there exists $r$ such that $\G_{u;v}(r)=\G_{u';v'}(r)$ and $\G_{u;v}(r+1)\ne \G_{u';v'}(r+1)$.
    Assume for the sake of contradiction that $\G_{u;v}, \G_{u';v'}$ intersect above height $r+1$.
    Let $t$ be the first height at which such an intersection occurs, i.e. $t\coloneqq\min\bigl\{s>r+1:\G_{u;v}(s)=\G_{u';v'}(s)\bigr\}.$
    Then restricting $\G_{u;v}, \G_{u';v'}$ to the strip $\H_{\lb r,t\rb}$ produces two geodesics $(\G_{u;v}(r),r)\to(\G_{u;v}(t),t)$, one of which lies strictly to the left of the other (except at the starting and ending points).
    This contradicts uniqueness.
\end{proof}

The following lemma establishes positive temperature analogues of the above notions.
\begin{lemma}[Positive temperature polymer ordering and coalescence]\label{lem:polymer_ordering}
    Fix points $u=(x_1,t_1),u'=(y_1,t_1), v=(x_2,t_2),v'=(y_2,t_2)\in\H$ with $t_1<t_2$ and $x_i\le y_i$ for $i=1,2$.
    Let $\pi_{u;v}:u\to v$ and $\pi_{u';v'}:u'\to v'$ be polymers, i.e. paths distributed according to $\Q^{u;\,v}$ and $\Q^{u';\,v'}$, respectively.
    There exists a coupling of $\Q^{u;\,v}, \Q^{u';\,v'}$ under which the following hold:
    \begin{enumerate}[label=(\alph*)]
        \item \label{one} $\pi_{u;v}$ lies to the left of $\pi_{u';v'}$, and 
        \item \label{two} 
        $\pi_{u;v}\cap\pi_{u';v'}$ is a connected subset of $\H$.
    \end{enumerate}        
\end{lemma}
\begin{proof}
    We fix a realization of the environment $\o$, so that the only randomness in the following discussion is from the underlying random walk.

    We first construct a coupling with the desired properties in the case $v=v'$.
    Fix independent polymers 
    \[
        \pi_{u;v}:u\to v,\qquad\pi_{u';v}:u'\to v.
    \]
    We view $\pi_{u;v}, \pi_{u';v}$ as functions $\lb t_1,t_2\rb\to\Zpos$ and define a $\Zpos\times\Zpos$-valued process $\bm{\pi}$ by
    \[
        \bm{\pi}(t)\coloneqq
        \bigl(\pi_{u;v}(t),\pi_{u';v}(t)\bigr).
    \]
    Let $\GG=(\GG_t)_{t\in\lb t_1,t_2\rb}$ be the natural filtration induced by $\bm{\pi}$, and consider the $\GG$-stopping time
    \[
        \tau\coloneqq \inf\bigl\{t:\pi_{u;v}(t)=\pi_{u';v}(t)\bigr\}.
    \]
    Write
    \[
        W\coloneqq(\pi_{u;v}(\tau),\tau).
    \]
    Let $\pi_{u;v}^-$ be the segment of $\pi_{u;v}$ from $u$ to $W$, and let $\pi_{u;v}^+$ be the segment of $\pi_{u;v}$ from $W$ to $v$.
    Define $\pi_{u';v}^{\pm}$ analogously.
    It follows from the independence of $\pi_{u;v},\pi_{u';v}$ that $\bm{\pi}$ satisfies the strong Markov property.
    Therefore the pairs $(\pi_{u;v}^-, \pi_{u';v}^-)$ and $(\pi_{u;v}^+,\pi_{u';v}^+)$ are conditionally independent given $W$.
    Moreover,  $\pi_{u;v}^+$ and $\pi^+_{u';v}$ are conditionally independent given $\GG_\tau$, with the same conditional law (the polymer measure $\Q^{W;v}$). 
    Let $\wt\pi_{u;v}$ be the path $u\to v$ obtained from $\pi_{u;v}$ by replacing $\pi_{u;v}^+$ with $\pi_{u';v}^+$.
    By construction $\wt\pi_{u;v}$ lies to the left of $\pi_{u';v}$, and the preceding discussion ensures that the joint law of $(\wt\pi_{u;v}, \pi_{u';v})$ has marginals $\Q^{u;\,v}, \Q^{u';\,v}$.
    This is illustrated on the left side of \cref{fig:coalescence}.
    
    Assume now $v\ne v'$.
    Fix a sample $(\wt\pi_{u;v}, \pi_{u';v})$ from the coupling constructed above, as well as an independent sample $\pi_{u';v'}$.
    By reversibility of the random walk, we can view $\pi_{u';v'}$ as a sample from $\Q^{v';\,u'}$, the polymer measure defined in terms of paths $v'\to u'$ with steps in $\{(1,-1),(-1,-1)\}$.
    We can similarly view $\pi_{u';v}$ as a sample from $\Q^{v';\,u}$.
    The argument of the previous paragraph yields a path $\wt\pi_{u';v'}$ such that
    \begin{itemize}
        \item $\wt\pi_{u';v'}$ lies to the right of $\pi_{u';v}$,
        \item $\wt\pi_{u';v'}\cap\pi_{u';v}$ is connected, and
        \item the joint law of $(\wt\pi_{u';v'},\, \pi_{u';v})$ has marginals $\Q^{u';\,v'}, \Q^{u';\,v}$.
    \end{itemize} 
    By averaging over $\pi_{u';v}$, we conclude that $(\wt\pi_{u;v},\, \wt\pi_{u';v'})$ is a coupling of $\Q^{u;\,v}, \Q^{u';\,v'}$ under which properties \ref{one} and \ref{two} above hold (see \cref{fig:coalescence}).
\end{proof}

\begin{figure}[th]
    \centering
    
    \begin{subfigure}[t]{0.3\textwidth}
        \centering
        \begin{tikzpicture}[scale=3,
        dot/.style = {circle, draw=black, fill, minimum size=#1,
                      inner sep=0pt, outer sep=0pt},
        dot/.default = 4pt
        ]  
    
        \begin{pgfonlayer}{dots}
            \node[dot,anchor=center,label=below left:$u$](u) at (0,0)    {};
            \node[dot,anchor=center,label=above left:$v$](v) at (0,2)    {};
            \node[dot,anchor=center,label=below right:$u'$](u') at (0.6,0)   {};
            \node[dot=0pt,anchor=center, above right=of u] (hit1) {};
            \node[dot=0pt,anchor=center,below right=of v] (hit2) {};
        \end{pgfonlayer}
        
        \draw[color={DarkOrchid2},ultra thick] (u) .. controls  ++(0.3,0.3) and ++(0.05, -0.2) .. (hit1) node[pos=0.3,label=left:$\pi^-_{u;v}$]{};
        \draw[color={DarkOrchid2},ultra thick] (hit1) .. controls  ++(-0.2,0.9) and ++(0.1,-0.4) .. (hit2);
        \draw[color={DarkOrchid2},ultra thick] (hit2) .. controls  ++(-0.1,0.3) and ++(0,-0.2) .. (v);
    
        \draw[color={Chartreuse3},ultra thick,transform canvas={xshift=1.5pt}] (u')++(-0.01,0) .. controls  ++(0,0.3) and ++(0.05,-0.2) .. (hit1.center) node[pos=0.5,label=right:$\pi^-_{u';v}$]{};
        \draw[color={Chartreuse3},ultra thick,transform canvas={xshift=1.5pt}] (hit1) .. controls  ++(-0.2,0.9) and ++(0.1,-0.4) .. (hit2)
        node[pos=0.4,label=right:$\pi^+_{u';v}$]{};
        \draw[color={Chartreuse3},ultra thick,transform canvas={xshift=1.5pt}] (hit2) .. controls  ++(-0.075,0.3) and ++(-0.05,-0.175) .. (v.center);
    
        \draw[color=black, thin] (hit1.east) -- ++(0.3,0.1) node[right]{$W$};
        \end{tikzpicture}
    \end{subfigure}
    ~
    \begin{subfigure}[t]{0.3\textwidth}
        \centering
        \begin{tikzpicture}[scale=3,
            dot/.style = {circle, draw=black, fill, minimum size=#1,
                          inner sep=0pt, outer sep=0pt},
            dot/.default = 4pt
            ]  
        
            \begin{pgfonlayer}{dots}
                \node[dot=0pt,anchor=center](u) at (0,0)    {};
                \node[dot,anchor=center,label=above left:$v$](v) at (0,2)    {};
                \node[dot,anchor=center,label=below right:$u'$](u') at (0.6,0)   {};
                \node[dot,anchor=center,label=above right:$v'$](v') at (0.5,2)  {};
                \node[dot=0pt,anchor=center, above right=of u] (hit1) {};
                \node[dot=0pt,anchor=center,below right=of v] (hit2) {};
            \end{pgfonlayer}
        
            \draw[color={DarkOrange1},ultra thick,transform canvas={xshift=1.5pt}] (u')++(-0.01,0) .. controls  ++(0,0.3) and ++(0.05,-0.2) .. (hit1);
            \draw[color={DarkOrange1}, ultra thick,transform canvas={xshift=1.5pt}] (hit1) .. controls  ++(-0.2,0.9) and ++(0.1,-0.4) .. (hit2);
            \draw[color={DarkOrange1},ultra thick,transform canvas={xshift=1.5pt}] (v')++(-0.01,0) .. controls  ++(-0.1,-0.3) and ++(-0.05,0.2).. (hit2)
            node[pos=0.4,label=right:$\pi^-_{v';u'}$]{};
        
            \draw[color={Chartreuse3},ultra thick,line cap=round] (u') .. controls  ++(-0.03,0.285) and ++(0.05,-0.21) .. (hit1);
            \draw[color={Chartreuse3},ultra thick] (hit1) .. controls  ++(-0.2,0.9) and ++(0.1,-0.4) .. (hit2)
            node[pos=0.4,label=left :$\pi^+_{v;u'}$]{};
            \draw[color={Chartreuse3},ultra thick,line cap=round] (hit2) .. controls  ++(-0.1,0.3) and ++(0,-0.2) .. (v)
            node[pos=0.8,label=below:$\pi^-_{v;u'}$]{};
        \end{tikzpicture}
    \end{subfigure}
    ~
    \begin{subfigure}[t]{0.3\textwidth}
        \centering
        \begin{tikzpicture}[scale=3,
            dot/.style = {circle, draw=black, fill, minimum size=#1,
                          inner sep=0pt, outer sep=0pt},
            dot/.default = 4pt
            ]  
        
            \begin{pgfonlayer}{dots}
                \node[dot,anchor=center,label=below left:$u$](u) at (0,0)    {};
                \node[dot,anchor=center,label=above left:$v$](v) at (0,2)    {};
                \node[dot,anchor=center,label=below right:$u'$](u') at (0.6,0)   {};
                \node[dot,anchor=center,label=above right:$v'$](v') at (0.5,2)  {};
                \node[dot=0pt,anchor=center, above right=of u] (hit1) {};
                \node[dot=0pt,anchor=center,below right=of v] (hit2) {};
            \end{pgfonlayer}

            \draw[color={DarkOrchid2},ultra thick] (u) .. controls  ++(0.3,0.3) and ++(0.05, -0.2) .. (hit1);
            \draw[color={DarkOrchid2},ultra thick] (hit1) .. controls  ++(-0.2,0.9) and ++(0.1,-0.4) .. (hit2)
            node[pos=0.5,label=left :$\wt{\pi}_{u;v}$]{};
            \draw[color={DarkOrchid2},ultra thick] (hit2) .. controls  ++(-0.1,0.3) and ++(0,-0.2) .. (v);
        
            \draw[color={DarkOrange1},ultra thick,transform canvas={xshift=1.5pt}] (u')++(-0.01,0) .. controls  ++(0,0.3) and ++(0.05,-0.2) .. (hit1);
            \draw[color={DarkOrange1}, ultra thick,transform canvas={xshift=1.5pt}] (hit1) .. controls  ++(-0.2,0.9) and ++(0.1,-0.4) .. (hit2)
            node[pos=0.4,label=right :$\wt{\pi}_{u';v'}$]{};
            \draw[color={DarkOrange1},ultra thick,transform canvas={xshift=1.5pt}] (v')++(-0.01,0) .. controls  ++(-0.1,-0.3) and ++(-0.05,0.2).. (hit2);
        \end{tikzpicture}
    \end{subfigure}

        \caption{\textbf{Left:} The polymers $\pi_{u;v}$ (purple) and $\pi_{u';v}$ (green) begin their respective journeys at $u$ and $u'$.
            Their first intersection point is labeled $W$.
            The trajectory of $\pi_{u;v}$ after $W$ has been replaced by the trajectory of $\pi_{u';v}$ after $W$. \protect\footnotemark\;
            \textbf{Middle:} The green path $\pi_{u';v}$ appearing in the left figure is identified with a sample $\pi_{v;u'}$ (also green) from the time-reversed polymer measure.
            The polymer $\pi_{v';u'}$ (orange) begins its journey at $v'$ and at some (unlabeled) point hits $\pi_{v;u'}$.
            The remaining trajectory of $\pi_{v';u'}$ has been replaced by that of $\pi_{v;u'}$. \textsuperscript{\ref{foot}}\;
            \textbf{Right:} We superimpose the left and middle figures, omitting the green path $\pi_{u';v}=\pi_{v;u'}$.
            The desired coupling is obtained by identifying the orange path with its time-reversal $u'\to v'$ and averaging over the (now hidden) green path. \textsuperscript{\ref{foot}}}
        \label{fig:coalescence}
\end{figure}


\subsection{Correspondence between positive temperature and zero temperature}\label{sec:pos-zero}
Much of our analysis will apply simultaneously to the polymer free energy and the last passage time.
Let us make explicit the relationship between the two.

For $\beta>0$ (to be thought of as inverse temperature), consider the partition function $Z_\beta\coloneqq \sum_\pi e^{\beta H(\pi)}$, the free energy $F_\beta\coloneqq \beta^{-1}\log Z_\beta$, and the polymer measure $\Q_\beta(\{\pi\})\coloneqq  Z_\beta^{-1} e^{\beta H(\pi)}$.
In the zero temperature limit $\beta\to\infty$, we have $F_\beta\to L$ and (formally) $\Q_\beta\to\d_{\{\G\}}$, the Dirac mass on the left-most geodesic $\G$.

\footnotetext{In \cref{fig:coalescence} each pair of paths is drawn slightly shifted so as to make both visible, even when they have coalesced. This is unrealistic, but is intended to clarify the dual roles played by the paths in the proof of \cref{lem:polymer_ordering}.\label{foot}}

We have omitted $\beta$ from our definitions in \cref{sec:model}, as it can be absorbed into the weights $\bkwt,\vtwt$.
Accordingly, we replace the formal zero temperature limit with the following ``tropicalization'' correspondence: for any set of paths $A$ and any $c\in[0,1)$, we write
\begin{equation}\label{eq:dictionary}
    \begin{split}
        \Q \quad&\longmapsto\quad \d_{\{\G\}},\\
        \{\o:\Q(A)>c\} \quad &\longmapsto \quad
        \{\o:\text{ the left-most geodesic } \G\in A\},\\
        \log\left(\sum_{\pi\in A}e^{H(\pi)}\right)\quad &\longmapsto \quad \sup_{\pi\in A} H(\pi),\\
        Z\quad &\longmapsto \quad e^{L}.
        \end{split}
\end{equation}
We will use the above dictionary to streamline our presentation in the following manner.
All the results below apply to the zero and positive temperature models simultaneously, and in our proofs we always treat the positive temperature model first.
We will then be able to convert the proof of the positive temperature statement into a proof of the zero temperature statement by formally replacing all instances of the symbols on the left side of \eqref{eq:dictionary} that appear in the positive temperature proof with the corresponding symbols on the right side of \eqref{eq:dictionary}.
With that said, merely appealing to this dictionary in the more complicated proofs (especially those in \cref{sec:approx_indep}) would demand an unreasonable amount of bookkeeping from the reader---in such cases we present the details of the zero temperature argument, referencing the dictionary to expedite the discussion when appropriate.


\section{Pinning}\label{sec:pinning}
In this section we prove \cref{thm:pinning}, thereby establishing that LLN separation implies the pinning of the polymer to $\V$.

Fix $t_1<t_2$ and $x_1,x_2\ge 0$.
We denote by $\Exc(x_1,t_1;x_2,t_2)$ the set of \emph{excursions} $(x_1,t_1)\to(x_2,t_2)$, i.e. paths that do not hit $\V$ (unless $x_1=0$ or $x_2=0$):
\[
    \Exc(x_1,t_1;x_2,t_2)\coloneqq \bigl\{\pi\in\Pi(x_1,t_1;x_2,t_2):\pi \cap \V \subset \{(0,t_1),(0,t_2)\}\bigr\}.
\]
The following lemma asserts that, under LLN separation, excursions typically are not competitive with paths that hit $\V$.

\begin{lemma}[Excursions are rare]\label{lem:pin_lemma}
    There exist constants $C,C'>0$ and $k_0\ge1$ depending only on the law of $\o$ such that the following holds.
    Fix $G\in\{F,L\}$. Suppose the polymer model has LLN separation, i.e. $\gap>0$.
    Fix $(x_1,t_1), (x_2,t_2)$ satisfying $t_2-t_1 \ge k_0(x_1+x_2+1)$ and $\Pi(x_1,t_1;x_2,t_2)\ne\varnothing$.
    If $G=F$ then
    \begin{equation}\label{eq:pin1}
        \P[\Q^{(x_1,t_1;\,x_2,t_2)}\left(\Exc(x_1,t_1;x_2,t_2)\right) > 
        e^{-\gap(|t_2-t_1|+|x_2-x_1|)}
        ]\le C\exp\left(-C'|t_2-t_1|^{1/3}\right).
    \end{equation}
    If $G=L$ and we denote by $\G$ the leftmost geodesic $(x_1,t_1)\to(x_2,t_2)$, then
    \begin{equation}\label{eq:pin2}
        \P\bigl(\G\in\Exc(x_1,t_1;x_2,t_2)\bigr)\le C\exp\left(-C'|t_2-t_1|^{1/3}\right).
    \end{equation}
\end{lemma}

\begin{proof}
    We will prove \eqref{eq:pin1}.
    The correspondence of \cref{sec:pos-zero} will immediately yield \eqref{eq:pin2}.

    Let us first record two properties of the bulk LLN $f^\bk$.
    By superadditivity, \eqref{eq:asymp_free_energy_bulk}, and \cref{lem:vertical_maximizes_g}, we have
    \begin{equation}\label{eq:superadd}
        \E[F^\bk(x_1,t_1;x_2,t_2)]\le f^\bk\cdot(|t_2-t_1|+|x_2-x_1|).
    \end{equation}
    By LLN separation \eqref{eq:DEF_lln_separation}, \eqref{eq:gap}, there exists $k_0'>0$ such that
    \begin{equation}\label{eq:lln_separation}
        \bE[F(0,0;0,t)]\ge (f^\bk+4\gap)\,t\quad\text{for all even } t\ge k_0'.
    \end{equation}

    We now set $k_0\coloneqq k_0'+c_0+2$, where $c_0$ is from \cref{lem:large_dev}.
    By vertical translation-invariance, it suffices to prove \eqref{eq:pin1} for $\Q^{(x_1,0;\,x_2,t)}$ with $t,x_1,x_2\ge0$ satisfying $t\ge k_0(x_1 + x_2+1)$.

    For simplicity we first treat the case $x_1=x_2=0$.
    We suppress the superscripts from the polymer measure and rewrite the left side of \eqref{eq:pin1} as
    \begin{equation}\label{eq:poly_prob_free_energy}
        \begin{split}
            \P\bigl(\Q(\Exc(0,0;0,t)) > e^{-\gap t}\bigr)
            &=\P\bigl(\log \Q(\Exc(0,0;0,t)) > -\gap t\bigr)\\
            &=\P\bigl(F^\exc(0,0;0,t) > F(0,0;0,t) - \gap t\bigr),
        \end{split}
    \end{equation}
    where we define $F^\exc(x_1,0;x_2,t) \coloneqq \log\left(\sum_{\pi\in\Exc(x_1,0;x_2,t)}e^{H(\pi)}\right)$.
    Since a path $\pi\in\Exc(0,0;0,t)$ only collects bulk weights $\o^\bk(x,s)$ (except at its endpoints $\pi(0),\pi(t)$), we have the estimate
    \[
        F^\exc(0,0;0,t) \le F^\bk(1,1;1,t-1) + \o(0,0)+\o(0,t).
    \]
    On the other hand, we have that $F(0,0;0,t) = F(1,1;1,t-1) + \o(0,0)+\o(0,t)$.
    Write $F\coloneqq F(1,1;1,t-1)$ and $F^\bk\coloneqq F^\bk(1,1;1,t-1)$.
    By substituting the above display into \eqref{eq:poly_prob_free_energy} and applying \eqref{eq:superadd}, \eqref{eq:lln_separation}, and \cref{lem:large_dev} (note that $t-2\ge k_0-2 \ge c_0$), we conclude that
    \begin{equation}\label{eq:lln_estimate}
        \begin{split}
            \P\bigl(\Q(\Exc(0,0;0,t)) > e^{-\gap t}\bigr)
            &\le \P[F< F^\bk + \gap t]\\
            &\le \P[F < (f^\bk+\gap)(t-2)  + \gap t]\\
            &\qquad\quad+ \P[F^\bk > (f^{\bk} + \gap)(t-2)]\\
            &\le \P\bigl(F < \bE[F]-2\gap (t-2) + 2\gap\bigr)
            \,+ C\exp\bigl(-C't^{1/3}\bigr)\\
            &\le \P\bigl(F < \bE[F]-\gap (t-2)\bigr)
            \,+ C\exp\bigl(-C't^{1/3}\bigr)\\
            &\le C\exp\bigl(-C't^{1/3}\bigr).
        \end{split}
    \end{equation}
    Here we absorbed the factor $\min\{1,\gap^2\}$ from \cref{lem:large_dev} into the constant $C'$.

    Suppose now that $x_1>0$ and $x_2>0$.
    It follows that
    \[
        F^\exc(x_1,0;x_2,t) \le F^\bk(x_1,0;x_2,t).
    \]
    On the other hand, by superadditivity and the fact that the free energies are positive almost surely,
    \[
        \begin{split}
            F(x_1,0;x_2,t) &\ge F(x_1,0;1,x_1-1) + F(0,x_1;0,t-x_2) + F(1, t-x_2+1; x_2,t)\\
            &\ge F(0,x_1;0,t-x_2).
        \end{split}
    \]
    We also note that by applying \eqref{eq:lln_separation} and increasing $k_0$ as needed, we can assume that
    \[
        (f^\bk+2\gap)(t+|x_2-x_1|)\le \bE[F(0,x_1;0,t-x_2)] - \gap(t-x_2-x_1)\quad\text{whenever}\quad t\ge k_0(x_1+x_2+1).
    \]
    In combination with the above three displays, the condition $\Q(\Exc(x_1,0;x_2,t))>e^{-\gap(t+|x_2-x_1|)}$ implies a comparison of 
    $F^\bk\coloneqq F^\bk(x_1,0;x_2,t)$ and $F\coloneqq F(0,x_1;0,t-x_2)$, as in \eqref{eq:poly_prob_free_energy}.
    We obtain, as in \eqref{eq:lln_estimate},
    \[
        \begin{split}
            \P[\Q(\Exc(x_1,0;x_2,t)) > e^{-\gap(t+|x_2-x_1|)}]
            &\le \P[F < F^{\bk} +\gap(t+|x_2-x_1|)]\\
            &\le \P[F<(f^\bk + 2\gap)(t+|x_2-x_1|)]\\
            &\qquad\qquad + \P[F^\bk > (f^\bk+\gap)(t+|x_2-x_1|)]\\
            &\le \P\bigl(F<\bE[F] - \gap(t-x_2-x_1)\bigr)
            + C\exp\bigl(-C't^{1/3}\bigr)\\
            &\le C\exp\bigl(-C'(t-x_2-x_1)^{1/3}\bigr)\\
            &\le C\exp\bigl(-C't^{1/3}\bigr).
        \end{split}
    \]
    
    The cases $x_1=0,x_2>0$ and $x_1>0,x_2=0$ follow from a straightforward combination of the previous two arguments.
    We omit the details.
\end{proof}

We now deduce \cref{thm:pinning} from \cref{lem:pin_lemma}.

\begin{proof}[Proof of \cref{thm:pinning}]
    We will treat the case $G=F$.
    The case $G=L$ will then follow from the correspondence of \cref{sec:pos-zero}.

    Fix $s_1,s_2\in\lb x_1+1,\,t-x_2-1\rb$ with $s_2-s_1\ge k_0$.
    Notice that if a path $\pi:(x_1,0)\to(x_2,t)$ does not intersect $\V_{\lb s_1,\,s_2\rb}$, then there exist $a\in\lb 0,\,s_1-1\rb$ and $b\in\lb s_2+1,\,t\rb$ such that 
    \[
        \pi\cap\H_{\lb a,b\rb} \in \Exc(\pi(a),a;\pi(b),b).
    \]
    We denote by $a_\pi$ the minimal such $a$, and by $b_\pi$ the maximal such $b$.
    Observe that $\pi(a_\pi)\in\{0,x_1\}$ and $\pi(b_\pi)\in\{0,x_2\}$.
    Moreover, $\pi(a_\pi)\ne 0$ only if $a_\pi=0$, and $\pi(b_\pi)\ne 0$ only if $b_\pi=t$.
    The idea now is to perform a union bound over the possible pairs $(a_\pi,b_\pi)$ and apply \cref{lem:pin_lemma} to control the tails.
    To this end we record the following estimates for the relevant polymer measures.
    
    Suppose $0<a\le b<t$. From the inequality
    \[
        Z(x_1,0;x_2,t)
        \ge Z(x_1,0;1,a-1)\, Z(0,a;0,b)\,Z(b+1,1;x_2,t),
    \]
    we deduce that
    \[
        \Q^{(x_1,0;\,x_2,t)}\left(\pi\cap\H_{\lb a,\,b\rb}\in\Exc(0,a;0,b)\right) \le \Q^{(0,a;\,0,b)}\left(\Exc(0,a;0,b)\right).
    \]
    A similar argument shows that
    \[
        \Q^{(x_1,0;\,x_2,t)}\left(\pi\cap\H_{\lb 0,\,b\rb}\in\Exc(x_1,0;0,b)\right) \le \Q^{(x_1,0;\,0,b)}\left(\Exc(x_1,0;0,b)\right),
    \]
    and that
    \[
        \Q^{(x_1,0;\,x_2,t)}\left(\pi\cap\H_{\lb a,\,t\rb}\in\Exc(0,a;x_2,t)\right) \le \Q^{(0,a;\,x_2,t)}\left(\Exc(0,a;x_2,t)\right).
    \]

    Now we perform the union bound.
    In particular, observe that by choosing $\d=\d(k_0,\gap)>0$ sufficiently small and $C''=C''(k_0,\gap)>0$ sufficiently large
    (each depending only on $k_0, \gap$), we can ensure that the following estimate holds for any $s_1,s_2$ satisfying the above hypotheses:
    \[
        \begin{split}
            \P\biggl(\Q^{(x_1,0;\,x_2,t)}\bigl(\pi\cap\V_{\lb s_1,\,s_2\rb}&=\varnothing\bigr)
            >C'' e^{-\d\gap |s_2-s_1|}\biggr)\\
            &\le \sum_{\substack{a<s_1,\\b>s_2}}
            \P[\Q^{(0,a;\,0,b)}(\Exc(0,a;0,b))>e^{-\gap|b-a|}]\\
            &\quad + \sum_{a<s_1}\P[\Q^{(0,a;\,x_2,t)}\left(\Exc(0,a;x_2,t)\right)>e^{-\gap (|t-a| + x_2)}]\\
            &\quad +\sum_{b>s_2}\P[\Q^{(x_1,0;\,0,b)}\left(\Exc(x_1,0;0,b)\right)>e^{-\gap (b+x_1)}]\\
            &\quad + \P[\Q^{(x_1,0;\,x_2,t)}(\Exc(x_1,0;x_2,t))>e^{-\gap (t+|x_2-x_1|)}].
        \end{split}
    \]
    Applying \cref{lem:pin_lemma} to the terms on the right side of the above display yields \cref{thm:pinning}.
\end{proof}

We conclude this section by recording two consequences of \cref{thm:pinning} that will be useful later.
First is the following lemma, which asserts that the polymer has $O(1)$ transversal fluctuations under LLN separation. It is essentially duplicated from \cref{rem:reformulating_pinning}, and we omit the proof. 
\begin{lemma}[$O(1)$ transversal fluctuations]\label{lem:O1fluc}
    There exist constants $C,C',C'',C'''>0$ and $k_0\ge1$ depending only on the law of $\o$ such that the following hold.
    Fix $G\in\{F,L\}$ and suppose the polymer model has LLN separation.
    Fix $t,x_1,x_2\ge0$ satisfying $t\ge k_0(x_1+x_2+1)$ and $\Pi(x_1,0;x_2,t)\ne\varnothing$.
    Also fix $s\in\lb x_1,\,t-x_2\rb$.
    If $G=F$ then for all $k\ge k_0$,
    \[
        \P[\Q^{(x_1,0;\,x_2,t)}(\pi(s)>k)>C''e^{-C'''k}] \le C\exp\left(-C'k^{1/3}\right).
    \]
    If $G=L$ and we denote by $\G$ the left-most geodesic $(x_1,0)\to(x_2,t)$, then for all $k\ge k_0$,
    \[
        \P(\G(s)>k) \le C\exp\left(-C'k^{1/3}\right).
    \]
\end{lemma}

The next lemma is the result of performing a union bound in the conclusion of \cref{thm:pinning}.
\begin{lemma}[The polymer hits $\V$ near the beginning and end of its journey]\label{lem:no_avoiding_initial_segment}
    There exist constants $C,C',C'',C'''>0$ and $k_0\ge1$ depending only on the law of $\o$ such that the following hold.
    Fix $G\in\{F,L\}$ and suppose the polymer model has LLN separation.
    Fix $t,x\ge1$ satisfying $t\ge 2k_0(x+1)$ and $\Pi(x,0;x,t)\ne\varnothing$.
    Fix also $k\in\lb x+k_0,\,t-x-k_0\rb$.
    If $G=F$ then
    \[
        \P[\Q^{(x,0;\,x,t)}
        \left(\{
            \pi:\pi\cap\V_{\lb 0,k\rb}=\varnothing\ \ \text{or}\ \ \pi\cap\V_{\lb t-k, t\rb}=\varnothing
        \}\right) >
        C''e^{-C''' k}]
        \le C\exp(-C'k^{1/3}).
    \]
    If $G=L$ and we denote by $\G$ the left-most geodesic $(x,0)\to(x,t)$, then
    \[
        \P[\G\cap\V_{\lb 0,\,k\rb}=\varnothing\ \ \text{or}\ \ \G\cap\V_{\lb t-k,\,t\rb}=\varnothing]
        \le C\exp(-C'k^{1/3}).
    \]
\end{lemma}


\section{Variance grows linearly}\label{sec:linear_variance}
In this section we show that $\Var(G(0,0;0,n))\asymp n$.

\subsection{Variance grows at least linearly}\label{sec:lower_bound}

To show that $\Var(G(0,0;0,n))\gs n$, we apply a general estimate due to \cite{NPDivergenceShapeFluctuations1995}.
Let us set up some notation.

Fix an integer $x_0\ge 0$ and let $B_0=B_0(x_0)>0$ be such that
\begin{equation}\label{eq:B1_asp}
    \P[\sup_{x\in\lb0,x_0\rb}\o(x,j)\le B_0] \ge 0.9\quad\text{for all } j\in\Z.
\end{equation}
As $\bkwt,\vtwt$ have unbounded supports (\cref{sec:model}\ref{P3}), we have that
\[
    q\coloneqq \P[\sup_{x\in\lb0,x_0\rb}\o(x,j) \ge B_0+1] > 0.
\]
Fix $a<b$ and $j\in\lb a,b\rb$.
For an environment $\o$ and a real number $B>0$, let $\o^B_j=(\o^B_j(x,t))_{(x,t)\in\H}$ be the environment obtained from $\o$ by replacing $\o(x,j)$ with $B$, for all $x\in\lb0,x_0\rb$.
Viewing $G(0,a;0,b)$ as a function of the environment, we set
\[
    \Delta_j(\o)\coloneqq \inf_{B\ge B_0+1}G(\o^B_j)-\sup_{0<B\le B_0}G(\o^{B}_j).
\]

The following is a special case of \cite[Theorem 8]{NPDivergenceShapeFluctuations1995}.
\begin{theorem}[{\cite[Theorem 8]{NPDivergenceShapeFluctuations1995}}]\label{thm:newman-piza}
    Fix $a<b$ with $\Pi(0,a;0,b)\ne\varnothing$, and fix  $G\in\{F,L\}$.
    Then, given $\e>0$ and subevents $\msf{F}_j\subset\{\o:\Delta_j(\o)\ge \e\}$ for $j\in\lb a,b\rb$, we have that
    \[
        \Var(G(0,a;0,b)) \ge 0.9\cdot q\cdot \e^2\sum_{j=a}^{b}\P(\F_j)^2.
    \]
\end{theorem}

We deduce the desired lower bound from \cref{thm:newman-piza} and the results of \cref{sec:pinning}:
\begin{lemma}[Variance grows at least linearly]\label{lem:varatleastlinear}
    There exists a constant $C>0$ depending only on the law of $\o$, such that for all $n\ge2$,
    \[
        \Var(G(0,0;0,n)) \ge Cn.
    \]
\end{lemma}

\begin{proof}
Suppose first $G=F$.
Write $\Q\coloneqq \Q^{(0,0;\,0,n)}$.
By \cref{lem:O1fluc}, there exists $x_0\ge 0$ such that for any $j\in\lb 1,n\rb$,
\begin{equation}\label{eq:Aj}
    \text{the event}\quad
    \A_j\coloneqq \Bigl\{\o:\Q\bigl(\pi(j)\in \lb 0, x_0\rb\bigr)\ge 0.1\Bigr\}\quad\text{satisfies}\quad \P(\A_j)\ge0.9.
\end{equation}
Then by \eqref{eq:B1_asp},
\[
    \text{the event}\quad
    \F_j\coloneqq \A_j\cap\left\{\o:\sup_{x\in\lb 0,x_0\rb}\o(x,j) \le B_0\right\}\quad\text{satisfies}\quad \P(\F_j) \ge 0.8.
\]
We write $B\coloneqq B_0+1$.
We fix $\o\in\F_j$ and denote by $H$ (resp. $H^B_j$) the Hamiltonian in the environment $\o$ (resp. $\o^B_j$).
Viewing the partition function $Z$ as a function of the environment, we have, by Jensen's inequality,
\[
\begin{split}
    \log Z(\o^B_j) - \log Z(\o)
    &= \log\left(\sum_{\pi}\frac{e^{H(\pi)}}{Z(\o)}e^{H^B_j(\pi)-H(\pi)}\right)\\
    &\ge \sum_{\pi}\frac{e^{H(\pi)}}{Z(\o)}(H^B_j(\pi)-H(\pi))\\
    &\ge \sum_{\pi:\pi(j)\in\lb 0,\,x_0\rb}\frac{e^{H(\pi)}}{Z(\o)}(B-B_0)\\
    &\ge 0.1.
\end{split}    
\]
We now apply \cref{thm:newman-piza} with $\e=0.1$ to obtain
\[
    \Var(F) \ge 0.9\cdot q\cdot 0.1^2\sum_{j=0}^n \P(\F_j)^2 \ge (0.00576\,q)\cdot n,
\]
which is just \cref{lem:varatleastlinear} for $G=F$.

The same argument applies at zero temperature.
Let all notation be as in the previous paragraph.
Fix $\o\in\F_j$ and let $\G:(0,0)\to(0,n)$ be a geodesic in $\o$, i.e. $L(\o)=H(\G)$.
Then we have
\[
        L(\o^B_j) - L(\o)
        \ge H^B_j(\G) - H(\G)
        \ge 1.
\]
\cref{thm:newman-piza} now implies
\cref{lem:varatleastlinear} for $G=L$.
\end{proof}

\subsection{Variance grows at most linearly}\label{sec:upper_bound}
We show that $\Var(G(0,0;0,n))\ls n$ using the Efron--Stein inequality, which we now recall (e.g. \cite[Lemma 3.2]{ADH50YearsFirstpassage2017}).
\begin{theorem}[Efron--Stein inequality]\label{thm:efron-stein}
    Let $\xi_1,\dots,\xi_n,\xi_1',\dots,\xi_n'$ be independent random variables with $\xi_i\law\xi_i'$ for all $i$.
    Then for a square-integrable function $h(\xi_1,\dots,\xi_n)$,
    \[
        \Var(h(\xi_1,\dots,\xi_n)) \le \frac{1}{2}\sum_{i=1}^n \E[(h(\xi_1,\dots,\xi_n)-h(\xi_1,\dots,\xi_{i-1}, \xi_i',\xi_{i+1}, \dots,\xi_n))^2].
    \]
\end{theorem}

We now establish our upper bound.
\begin{lemma}[Variance grows at most linearly]\label{lem:varatmostlinear}
    There exists a constant $C>0$ depending only on the law of $\o$, such that for all $n\ge2$,
    \[
        \Var(G(0,0;0,n)) \le Cn.
    \]
\end{lemma}
\begin{proof}
For $i\in\lb 1,n\rb$ we set $\o_i\coloneqq (\o(x,i))_{x\ge 0}$, so that the restriction of the environment $\o$ to $\H_{\lb 1,n\rb}$ is the tuple $W\coloneqq (\o_1,\dots,\o_n)$.
We fix independent copies $\o_i'\law\o_i$ for $i\in\lb 1,n\rb$, and write $W'_i\coloneqq (\o_1,\dots,\o_{i-1}, \o_i',\o_{i+1}, \dots,\o_n)$.

Suppose first $G=F$.
Fix $i\in\lb 1,n\rb$.
The polymer $\pi:(0,0)\to(0,n)$ induces a probability distribution on the horizontal line $\Zpos\times\{i\}$.
We denote by $x_i=x_i(W)$ a median of this distribution, i.e.
\[
    \Q^{(0,0;\,0,n)}\bigl(\pi(i)\le x_i\bigr) \ge \frac{1}{2}
    \quad\text{and}\quad
    \Q^{(0,0;\,0,n)}\bigl(\pi(i)\ge x_i\bigr) \ge \frac{1}{2}.    
\]
Let $x_i'$ denote the same with respect to the environment $W'_i$.

Let $H$ be the Hamiltonian in $W$ and $H'_i$ be the Hamiltonian in $W_i'$.
Viewing $F(0,0;0,n)$ as a function of $W$, we have
\[
    \begin{split}
        F(W) &\le \log\left(\sum_{\pi:\pi(i)\le x_i}\exp\left(H(\pi)\right)\right) - \log \frac12\\
        &\ls \log\left(\sum_{\pi:\pi(i)\le x_i}\exp\left(H_i'(\pi)\right)\cdot
         \exp\left(\sup_{x\in\lb 0,x_i\rb}\o(x,i)\right)\right)\\
        &= \log\left(\sum_{\pi:\pi(i)\le x_i}\exp\left(H_i'(\pi)\right)\right) + \sup_{x\in\lb 0,x_i\rb}\o(x,i)\\
        &\le F(W'_i) + \sup_{x\in\lb 0,x_i\rb}\o(x,i).
    \end{split}
\]
Interchanging the roles of $W,W'_i$, we conclude
\begin{equation}\label{eq:var_upper_1}
    |F(W)-F(W'_i)|\ls \sup_{x\in\lb 0,x_i\rb}\o(x,i) + \sup_{x\in\lb 0,x_i'\rb}\o'(x,i).
\end{equation}
Write $M(y)\coloneqq \sup_{x\in\lb 0,y\rb}\o(x,i)$ and $M'_i(y)\coloneqq \sup_{x\in\lb 0,y\rb}\o'(x,i)$.
Then we have 
\begin{align*}
        \E\bigl[(F(W)-F(W'_i))^2\bigr]
        &\le C\E[\left(\sum_{y=0}^\infty(M(y)+M_i'(y))\1_{\{x_i\ge y \text{ or } x_i'\ge y\}}\right)^2]\\
        &\le C \E[\left(\sum_{y=0}^\infty M(y)\1_{\{x_i\ge y \text{ or } x_i'\ge y\}}\right)^2]\\
        &\le C \sum_{y=0}^\infty\sum_{z=0}^y\E[M(y)M(z)\1_{\{x_i\ge y \text{ or } x_i'\ge y\}}]\stepcounter{equation}\tag{\theequation}\label{eq:var_upper_22}\\
        &\le C\sum_{y=0}^\infty y\,\E\bigl[M(y)^4\bigr]^{1/2}\,\P(x_i\ge y \text{ or } x_i'\ge y)^{1/2}
        \stepcounter{equation}\tag{\theequation}\label{eq:var_upper_2}\\
        &\le C \sum_{y=0}^\infty y(\log y)^{C'}\,\exp(-C''y^{1/3})
        \stepcounter{equation}\tag{\theequation}\label{eq:var_upper_3}\\
        &<\infty,
\end{align*}
where in \eqref{eq:var_upper_22} we expanded the square and used that $\{x_i\ge y\text{ or }x_i'\ge y\}\subset \{x_i\ge z\text{ or } x_i'\ge z\}$ for $y\ge z$,
in \eqref{eq:var_upper_2} we used the Cauchy--Schwarz inequality and the estimate $M(y)M(z)\le M(y)^2$ for $y\ge z$, and in \eqref{eq:var_upper_3} we used \cref{lem:O1fluc} and the fact that the fourth moment of the maximum of $y$-many i.i.d.  subexponential random variables grows polylogarithmically in $y$ (for us, all but $\o(0,i)$ are identically distributed, which is irrelevant asymptotically). 
The constants appearing in \eqref{eq:var_upper_3} do not depend on $i$.
Therefore there exists $C>0$ depending only on the law of $\o$ with
\[
    \sum_{i=1}^n \E[(F(W)-F(W_i'))^2] \le Cn\quad\text{for all } n\ge2.
\]
This and \cref{thm:efron-stein} together imply \cref{lem:varatmostlinear} for $G=F$.

The preceding argument also handles the zero temperature case $G=L$.
To see this, observe that the correspondence of \cref{sec:pos-zero} yields $x_i=\G(i)$. It follows that
\[
    L(W) \le H'_i(\G) + \o(x_i,i) \le L(W')+ \o(x_i,i).
\]
Interchanging the roles of $W,W'$ yields the estimate
\[
    |L(W)-L(W')|\le \o(x_i,i)+\o(x_i',i).
\]
This is \emph{stronger} than \eqref{eq:var_upper_1}, so the argument following \eqref{eq:var_upper_1} implies \cref{lem:varatmostlinear} for $G=L$.
\end{proof}


\section{Free energy is approximately a sum of independent random variables}\label{sec:approx_indep}

The main result of this section (\cref{thm:comparison}) asserts that the diffusively-scaled free energy is approximated by a sum of independent random variables.
We begin by defining the latter.

For each (even) $n\ge 4$, we fix $J,K$ satisfying
\begin{equation}\label{eq:JK}
    J^{1/4}\in [(\log n)^{5/4},\; 2(\log n)^{5/4}]\cap 2\Z
    \qquad\text{and}\qquad K\in [ n^{0.9},\;2n^{0.9}]\cap2\Z,
\end{equation}
where $2\Z$ denotes the set of even integers.
In particular,
\footnote{We choose the exponents $5$ and $0.9$ essentially arbitrarily (cf. the proof sketch in \cref{sec:outline}), with the sole purpose of improving readability.}
\[
    J\asymp (\log n)^5,\qquad K\asymp n^{0.9}.
\]
Let $N\coloneqq \sup\{i:iK+J\le n\}$.
Then $NK\asymp n$, i.e.
\[
    N\asymp n^{0.1}.
\]
For $i\in\lb 1,N\rb$ we define
\begin{equation}\label{eq:si_mi_ti}
    s_i\coloneqq iK,\quad m_i\coloneqq iK+J,\quad\text{and}\quad
    t_i\coloneqq (iK+2J) \wedge n.
\end{equation}
We also set $m_0\coloneqq 0$ and $m_{N+1}\coloneqq n$.
The letters $s,m,t$ respectively stand for ``start,'' ``middle,'' and ``terminal'' (see \cref{fig:local_highways}).
Fix $G\in\{F,L\}$.
We define
\begin{equation}\label{eq:def_Fi}
    G_i^0\coloneqq G(0,m_i;0,m_{i+1}) - \o(0,m_{i+1})\quad\text{for } i\in\lb 0,N-1\rb,
\end{equation}
and $G_N^0 \coloneqq  G(0,m_N;0,m_{N+1})$.

As alluded to above, the present section is aimed at proving the following theorem.

\begin{theorem}[Free energy is approximately a sum of independent random variables]\label{thm:comparison}
    As $n\to\infty$,
    \begin{equation}\label{eq:pos_temp_comparison}
        \frac{1}{\sqrt n}\left|F(0,0;0,n)-\sum_{i=0}^{N} F_i^0\right| \pto  0,
    \end{equation}
    and
    \begin{equation}\label{eq:zero_temp_comparison}
        \frac{1}{\sqrt n}\left|L(0,0;0,n)-\sum_{i=0}^{N} L_i^0\right| \pto  0.
    \end{equation}
\end{theorem}

We first prove the positive temperature result \eqref{eq:pos_temp_comparison}.
\subsection{Proof of \cref{thm:comparison}: positive temperature}\label{sec:poscomp}
Let us introduce some notation.
We denote by $\Pi^\cons$ the set of paths $\pi:(0,0)\to(0,n)$ satisfying $\pi(s_i),\pi(t_i)\in\lb 0,J^{1/2}\rb$ for all $i\in\lb 1,N\rb$ (the superscript ``$\cons$" is an abbreviation of ``constrained'').
Let $Z^\cons$ be the partition function with respect to $\Pi^\cons$, and let $F^\cons$ be the corresponding free energy:
\[
    Z^\cons\coloneqq \sum_{\pi\in \Pi^\cons}e^{H(\pi)},\qquad F^\cons\coloneqq  \log Z^\cons.
\]
Write $F\coloneqq F(0,0;0,n)$.
The following lemma, a direct consequence of \cref{lem:O1fluc}, shows that it suffices to prove \eqref{eq:pos_temp_comparison} with $F$ replaced by $F^\cons$.
\begin{lemma}[The polymer is constrained]\label{lem:cons_reduction}
    There exists an event $\A_1$ with $\P(\A_1)=1-o(1)$, such that on $\A_1$, it holds that $\Q^{(0,0;\,0,n)}(\Pi^\cons)= 1-o(1)$, i.e. $F = F^\cons + o(1)$.
\end{lemma}

For $i\in\lb1,N\rb$ we write
\begin{equation}\label{eq:hwy_def1}
    \Pi_i\coloneqq \Pi(J^{1/2},s_i;J^{1/2},t_i).
\end{equation}
We say that a path $\g_i\in\Pi_i$ is a \emph{local highway} if it satisfies
\[
    \g_i\cap\V_{\lb s_i,\,\, s_i+J^{3/4}\rb} \ne \varnothing
    \quad\text{and}\quad
    \g_i\cap\V_{\lb t_i-J^{3/4},\,\, t_i\rb}\ne\varnothing,
\]
and we denote by $\Pi_i^\hwy$ the set of local highways $\g_i$.
The polymer measure on $\Pi_i$ is given by
\[
    \Q_i(\{\g_i\})\coloneqq \frac{e^{H(\g_i)}}{\sum_{\g_i\in\Pi_i}e^{H(\g_i)}} 
    \,\,\text{ for } \g_i\in\Pi_i.
\]
Similarly to \cref{lem:cons_reduction}, an application of \cref{lem:no_avoiding_initial_segment} shows that, typically, every $\g_i\sim\Q_i$ is a local highway:
\begin{lemma}[Every $\g_i$ is a local highway]\label{lem:hwy_reduction}
    There exists an event $\A_2$ with $\P(\A_2)=1-o(1)$, such that on $\A_2$ it holds that
    \[
        \Q_i\left(\Pi_i^\hwy\right)
        \ge 1-C''e^{-C''' J^{3/4}}\quad\text{for all } i\in\lb1,N\rb,
    \]
    where $C'',C'''>0$ are constants depending only on the law of $\o$.
\end{lemma}
\begin{proof}
    By taking $n$ sufficiently large, we can apply \cref{lem:no_avoiding_initial_segment} with $x=J^{1/2}$,\; $k=J^{3/4}$, and $t=2J$.
    Doing so, we obtain
    \[
        \P[\Q_i
        \bigl(
            \Pi_i^\hwy
        \bigr) \ge 1-C''e^{-C''' J^{3/4}}]
        \ge 1- C\exp\left(-C'\,J^{1/4}\right)=
        1-C\exp\left(-C'\,(\log n)^{5/4}\right).
    \]
    The lemma now follows by taking a union bound over $i\in\lb 1,N\rb$. 
\end{proof}

Roughly speaking, we will use the local highways to establish decay of correlation for the polymer.
This will lead to a proof of \eqref{eq:pos_temp_comparison}.
For $i\in\lb1,N\rb$ we define
\begin{equation}\label{eq:def_xistar}
    x_i^*\coloneqq \argmax_{x_i\in\lb 0,2J\rb} \,
    \Q_i\left(\{\g_i(m_i)=x_i\}\cap\{\g_i\in\Pi_i^\hwy\}\right),
\end{equation}
with some arbitrary deterministic rule for breaking ties.
We also write $x_0^*\coloneqq 0$ and $x_{N+1}^*\coloneqq 0$.
We define, in analogy with \eqref{eq:def_Fi},
\[
    F_i^*\coloneqq F(x_i^*,m_i;x_{i+1}^*,m_{i+1})-\o(x_{i+1}^*,m_{i+1})\quad\text{for } i\in\lb 0,N-1\rb
\]
and $F_N^* \coloneqq  F(x_N^*,m_N;x_{N+1}^*,m_{N+1})$.
We set
\[
    F^*\coloneqq \sum_{i=0}^{N} F_i^*
    \quad\text{and}\quad F^0\coloneqq \sum_{i=0}^{N}F_i^0.
\]

The approximation \eqref{eq:pos_temp_comparison} is an immediate consequence of the following two lemmas.
\begin{lemma}\label{lem:argmax_reduction}
    As $n\to\infty$,
    \[
        \frac{1}{\sqrt n}|F-F^*|\pto 0.
    \]
\end{lemma}
\begin{lemma}\label{lem:iid_reduction}
    As $n\to\infty$,
    \[
        \frac{1}{\sqrt n}\left|F^*-F^0\right|\pto 0.
    \]
\end{lemma}

It remains to prove \cref{lem:argmax_reduction,lem:iid_reduction}.
Let $\Q^\cons$ be the polymer measure on $\Pi^\cons$, i.e. $\Q^\cons(\{\pi\})\propto e^{H(\pi)}$ for $\pi\in\Pi^\cons$.
Let $\pi\sim\Q^\cons$ be the corresponding polymer.
Consider also polymers $\g_i\sim\Q_i$ for each $i\in\lb1,N\rb$.
By the definitions of $\Pi^\cons$ and $\Pi_i$, we have that 
\[
    \pi(s_i)\le \g_i(s_i)\quad\text{and}\quad \pi(t_i)\le\g_i(t_i),\quad\text{for all } i\in\lb 1,\,N\rb.
\]
Therefore, fixing a realization of the environment $\o$, we can apply \cref{lem:polymer_ordering} conditionally given the points $\{\pi(s_i),\pi(t_i) : i\in\lb 1,N\rb\}$ to obtain a coupling of $\pi,(\g_i)_{i\in\lb 1, N\rb}$ under which $\pi$ lies to the left of $\g_i$ and $\pi\cap\g_i$ is connected, for all $i\in\lb 1,N\rb$.
Averaging over $\{\pi(s_i),\pi(t_i) : i\in\lb 1,N\rb\}$ then yields a coupling $\QQ$ of the unconditional measures $\Q^\cons$ and $\bigotimes_{i\in\lb1,N\rb}\Q_i$ with the same polymer ordering and coalescence properties.

We write
\[
    \BB_i\coloneqq  \lb 0, 2J\rb \times \lb s_i,t_i\rb\quad\text{for all } i\in\lb 1,N\rb.
\]
\begin{proof}[Proof of \cref{lem:argmax_reduction}]
By increasing $n$, we can assume that \emph{every} path $(J^{1/2},s_i)\to(J^{1/2},t_i)$ is contained in $\BB_i$.
By the pigeonhole principle and \eqref{eq:def_xistar},
\[
    \QQ\bigl(\g_i(m_i)=x_i^*\bigr)\ge \frac{1}{2J+1}\quad\text{for all } i\in\lb1,N\rb.
\]
Then by \cref{lem:hwy_reduction}, it holds for all $\o\in\A_2$ that
\begin{equation}\label{eq:pigeonhole1}
    \QQ\left(\{\g_i(m_i)=x_i^*\}\cap\{\g_i\in\Pi_i^\hwy\}\right) > \frac{1}{2J+1} - C''e^{-C''' J^{3/4}}.
\end{equation}
On the other hand, consider a sample $(\pi,(\g_i)_{i\in\lb1,N\rb})\sim\QQ$.
If $\g_i\in\Pi_i^\hwy$ then by planarity,
\[
    \g_i\cap\pi\cap \H_{\lb s_i,\,\,s_i+J^{3/4}\rb} \ne \varnothing
    \quad\text{and}\quad
    \g_i\cap\pi\cap \H_{\lb t_i-J^{3/4},\,\,t_i\rb} \ne\varnothing.
\]
Since $\pi,\g_i$ coalesce under $\QQ$, the above display implies that $\g_i(m_i)=\pi(m_i)$ (see also \cref{fig:local_highways}).
Therefore, from \eqref{eq:pigeonhole1} and the fact that the $\g_i$ are i.i.d., we conclude that for $\o\in\A_2$,
\begin{equation}\label{eq:pigeonhole}
    \QQ\bigl(\pi(m_i)=x_i^* \text{ for all } 
     i\in\lb1,N\rb\bigr)
     \ge \left(\frac{1}{2J+1} - C''e^{-C''' J^{3/4}}\right)^N.
\end{equation}
Taking logarithms in \eqref{eq:pigeonhole} and recalling that $J\asymp(\log n)^5$, we get
\[
    F^* \ge F^\cons - CN\log\log n.
\]
We also have the deterministic inequality $F^*\le F^\cons.$
Recalling that $N\asymp n^{0.1}$, we obtain
\[    
    \frac{|F^\cons-F^*|}{\sqrt n}\le C\frac{\log\log n}{n^{0.4}} = o(1).
\]
\cref{lem:argmax_reduction} now follows by applying \cref{lem:cons_reduction,lem:hwy_reduction}.
\end{proof}

\begin{figure}[t]
    \centering
    \begin{tikzpicture}[]
        \fill[black!15!white] (0,0.5) rectangle (4,4.5);
        \draw[color=gray, very thick, dotted](0,0.5)--(4,0.5)node[right]{$\lb 0,2J\rb\times\{s_i\}$};
        \draw[color=gray, very thick, dashed](0,2.5)--(4,2.5)node[right]{$\lb 0,2J\rb\times\{m_i\}$};
        \draw[color=gray, very thick,dotted](0,4.5)--(4,4.5)node[right]{$\lb 0,2J\rb\times\{t_i\}$};

        \draw[color=gray, thick, |<->|](-0.5,0.5)--(-0.5,1.25) node[midway, left]{$J^{3/4}$};
        \draw[color=gray, thick, |<->|](-0.5,3.75)--(-0.5,4.5) node[midway, left]{$J^{3/4}$};

        \draw[color=black, very thick] (0,-0.2)node[left]{$\V$} -- (0,5.1);
    
        \draw[color={DarkOrange1}, very thick]
        (0.3,-0.4)node[right]{$\pi$}
        to[out=110,in=-90] (0.1,0.67);

        \draw[color={DarkOrange1}, very thick]
        (0.1,4.35)to[out=90,in=-100] (0.2,5.3);
   
        \draw[blue,very thick]
        (0.25,0.5)node[below right]{$\g_i$}
        to[out=100,in=-45] (0,0.75) 
        to[out=75,in=-100] ++(0.2,0.2)
        to[out=100,in=-90]++(-0.2,0.3)
        to[out=100,in=-120]++(0.3,0.2)
        to[out=45,in=-80]++(0.3,0.2)
        to[out=100,in=-90]++(-0.6,0.5)
        to[out=80,in=-90] (0.5,2.5) 
        to[out=100,in=-45] ++(-0.2,0.3)
        to[out=140,in=-45] ++(-0.3,0.4)
        to[out=80,in=-100] ++(0.1,0.1)
        to[out=100,in=-70]++(-0.1,0.1)
        to[out=70,in=-100]++(0.2,0.25)
        to[out=90,in=-75] (0,4.2) 
        to[out=90,in=-120] (0.25,4.5);
    
        \draw[blue, fill] (0.25,0.5) circle [radius = 0.05];
        \draw [blue, fill] (0.25,4.5) circle [radius = 0.05];
        \draw[blue, fill] (0,0.75) circle [radius = 0.05];
        \draw [blue, fill] (0,4.2) circle [radius = 0.05];

\end{tikzpicture}
    
\caption{The box $\BB_i=\lb 0,2J\rb\times\lb s_i,t_i\rb$ is depicted as a shaded gray square. 
The polymer $\g_i$ (blue) hits the initial and final height-$J^{3/4}$ segments of $\BB_i\cap\V$, and is therefore a local highway: $\g_i\in\Pi_i^\hwy$.
By definition, the constrained polymer $\pi\sim\Q^\cons$ (orange) crosses the bottom dotted line segment (height $s_i$) between $\V$ and the starting point of $\g_i$, and similarly $\pi$ crosses the top dotted line segment (height $t_i$) between $\V$ and the ending point of $\g_i$.
Planarity entails that $\pi, \g_i$ intersect within the initial and final strips of height $J^{3/4}=o(J)$, which implies coalescence outside of these strips.
In particular, $\pi,\g_i$ coincide on the dashed midway line $\lb 0,2J\rb\times\{m_i\}$.
}
\label{fig:local_highways}
\end{figure}

We briefly explain the intuition for \cref{lem:iid_reduction}.
By \eqref{eq:def_xistar}, the vector $x^*=(x_1^*,\dots,x_N^*)$ depends only on the weights inside of $\BB\coloneqq \bigcup_{i\in \lb1,N\rb}\BB_i$.
Since the volume $|\BB| = o(\sqrt n)$, the polymer's behavior within $\BB$ is negligible on the diffusive scale.
It is therefore probabilistically inexpensive to replace each $(x_i^*,m_i)$ by a nearby deterministic point, namely $(0,m_i)$.

\begin{proof}[Proof of \cref{lem:iid_reduction}]
Let $\pi\sim \Q^\cons$ and let $\pi^0$ be sampled from the conditional measure
\[
    \pi^0\sim \Q^0(\cdot)\coloneqq \Q\bigl(\,\cdot\given[\pi^0(m_i)=0 \text{ for all } i\in\lb 1,N\rb]\bigr).
\]
We apply \cref{lem:polymer_ordering} conditionally given the points $\{\pi(m_i):i\in\lb 1,N\rb\}$ to obtain a coupling $\QQ'$ of $\Q^0,\Q^\cons$ under which $\pi^0$ lies to the left of $\pi$ and
\[
    \pi^0\cap\pi\cap\H_{\lb m_i,\;m_{i+1}\rb}\text{ is connected, for all } i\in\lb 1,N-1\rb.
\]
Our earlier construction of $\QQ$ can be lifted along with $\QQ'$ to a coupling $\QQ''$ of the measures $\Q^0,\Q^\cons,\bigotimes_{i\in\lb 1,N\rb}\Q_i$, i.e. of the triple $(\pi^0,\pi,(\g_i)_{i\in\lb 1,N\rb})$, under which
\begin{enumerate}[label=(Q\arabic*)]
    \item \label{Q1} $\pi^0$ lies to the left of $\pi$,
    \item \label{Q2} $\pi$ lies to the left of $\g_i$ for all $i\in\lb 1,N\rb$,
    \item \label{Q3} $\pi\cap\g_i$ is connected for all $i\in\lb 1,N\rb$, and
    \item \label{Q4} $\pi^0\cap\pi\cap\H_{\lb m_i,\;m_{i+1}\rb}$ is connected for all $i\in\lb 1,N-1\rb$.
\end{enumerate}

Fix $\o\in\A_2$.
Consider $(\pi^0,\pi,(\g_i)_{i\in\lb 1,N\rb})\sim\QQ''$.
Fix $i\in\lb 1,N-1\rb$.
By \cref{lem:hwy_reduction}, 
\begin{equation}\label{eq:overlap0}
    \QQ''\left(\g_i\in\Pi_i^\hwy \quad\text{and}\quad \g_{i+1}\in\Pi_{i+1}^\hwy\right)
    =1-o(1).
\end{equation}
Suppose that $\g_i\in\Pi_i^\hwy$ 
and $\g_{i+1}\in\Pi_{i+1}^\hwy$.
Then by planarity and \ref{Q2}, \ref{Q3} (cf. the above proof of \cref{lem:argmax_reduction}),
\begin{equation}\label{eq:overlap1}
    \pi\cap \V_{\lb t_i-J^{3/4},\,\,t_i\rb} \ne \varnothing
    \quad\text{and}\quad
    \pi\cap \V_{\lb s_{i+1},\,\,s_{i+1}+J^{3/4}\rb} \ne\varnothing.
\end{equation}
Therefore, by planarity and \ref{Q1},
\begin{equation}\label{eq:overlap2}
    \pi^0\cap\pi\cap \V_{\lb t_i-J^{3/4},\,\,t_i\rb} \ne \varnothing
    \quad\text{and}\quad
    \pi^0\cap\pi\cap \V_{\lb s_{i+1},\,\,s_{i+1}+J^{3/4}\rb} \ne\varnothing.
\end{equation}
Moreover, \eqref{eq:overlap1} implies $\pi(m_i)=x_i^*$ and $\pi(m_{i+1})=x_{i+1}^*$.
We rearrange \eqref{eq:overlap0} and sum over the possible intersection points in \eqref{eq:overlap1}, \eqref{eq:overlap2} to conclude that
\begin{equation}\label{eq:z0decomp}
    \exp(F_i^0) = (1+o(1))\sum_{a=t_i-J^{3/4}}^{t_i} \sum_{b=s_{i+1}}^{s_{i+1}+J^{3/4}} Z(0,m_i;1,a-1)\, Z(0,a;0,b)\, Z(1,b+1;0,m_{i+1}),
\end{equation}
and that
\begin{equation}\label{eq:z*decomp}
    \exp(F_i^*) = (1+o(1))\sum_{a=t_i-J^{3/4}}^{t_i} \sum_{b=s_{i+1}}^{s_{i+1}+J^{3/4}} Z(x_i^*,m_i;1,a-1)\, Z(0,a;0,b)\, Z(1,b+1;x_{i+1}^*,m_{i+1}).
\end{equation}
By \eqref{eq:z0decomp} and the pigeonhole principle, there exist
$a_i\in\lb t_i-J^{3/4},\;t_i\rb$ and $b_i\in \lb s_{i+1},\;s_{i+1}+J^{3/4}\rb$ such that
\begin{equation}\label{eq:pigeon1}
    \frac{1}{J^{3/2}}\exp(F^0_i) \ls  Z(0,m_i;1,a_i-1)\, Z(0,a_i;0,b_i)\, Z(1,b_i+1;0,m_{i+1}).
\end{equation}
On the other hand, as the summands in \eqref{eq:z*decomp} are nonnegative, we have
\begin{equation}\label{eq:pigeon2}
    \exp(F_i^*) \ge (1+o(1))Z(x_i^*,m_i;1,a_i-1)\, Z(0,a_i;0,b_i)\, Z(1,b_i+1;x_{i+1}^*,m_{i+1}).
\end{equation}
Let us emphasize that the \emph{same} factor $Z(0,a_i;0,b_i)$ appears in \eqref{eq:pigeon1} and \eqref{eq:pigeon2}---this can be interpreted as a manifestation of coalescence, cf. \ref{Q4}.
Combining \eqref{eq:pigeon1} and \eqref{eq:pigeon2} therefore yields 
\begin{equation*}
    \begin{split}
        F_i^0 - F_i^* &\ls F(0,m_i;1,a_i-1) + F(1,b_i+1;0,m_{i+1})
        \\
        &\qquad - F(x_i^*,m_i;1,a_i-1)-F(1,b_i+1;x_{i+1}^*,m_{i+1}) + \log J.
    \end{split}
\end{equation*}
Now by interchanging the roles of $F^0_i$ and $F^*_i$ in \eqref{eq:pigeon1} and \eqref{eq:pigeon2}, and recalling that $J\asymp(\log n)^5$, we conclude that
\begin{equation}\label{eq:pigeon3}
    \begin{split}
        |F_i^0-F_i^*| &\ls F(0,m_i;1,a-1)+ F(1,b+1;0,m_{i+1})\\
        &\qquad + F(x_i^*,m_i;1,a^*-1) +F(1,b^*+1;x_{i+1}^*,m_{i+1})   + \log \log n
\end{split}
\end{equation}
for some $a,a^*\in\lb t_i-J^{3/4},\,\,t_i\rb$ and some $b,b^*\in\lb s_{i+1},\,\, s_{i+1}+J^{3/4}\rb$.

We now observe that the terms on the right side of \eqref{eq:pigeon3} are typically of order $o\left(\sqrt{n}/N\right)$.
For instance, we have the deterministic inequality
\begin{equation}\label{eq:crude_upper}
    \begin{split}
        F(0,m_i;1,a-1) &\le L(0,m_i;1,a-1) + \log|\Pi(0,m_i;1,a-1)|\\
        &\le C(\log n)^{C'} \sup_{v\in\BB_i}\o(v) + C(\log n)^{C'}
    \end{split}
\end{equation}
for some absolute constants $C,C'$.
A union bound and the fact that the weights are subexponential 
(\cref{sec:model}\ref{P2})
shows that
\footnote{For our purposes the exponent $0.01$ can be replaced with any other constant $C>0$ satisfying $n^{C}=o\left(\sqrt{n}/N\right)$.}
\begin{equation}\label{eq:unionbound}
    \P[\sup_{v\in\BB_i}\o(v) > n^{0.01}] \le Ce^{-C'n^{0.01}}.
\end{equation}
Analogous estimates apply to the other terms on the right side of \eqref{eq:pigeon3}, and we conclude that there exists an event $\B_i$ with
\begin{equation}\label{eq:stretchedexpBi}
    \P(\B_i) \ge 1-Ce^{-C'n^{0.01}}
    \quad\text{and}\quad
    |F_i^0-F_i^*|=o\left(\sqrt n/N\right) \text{ on } \B_i\cap\A_2.
\end{equation}

As for the case $i\in\{0,N\}$, we note that since $\pi^0,\pi$ both start at $(0,0)$ and end at $(0,n)$, the coalescence argument above allows us to assume that $\pi^0,\pi$ coincide on $\H_{\lb 0,\, s_1\rb}\cup \H_{\lb t_N,\, n\rb}$.
This produces a decomposition analogous to \eqref{eq:z0decomp}, \eqref{eq:z*decomp}, but with the sum over only one intersection point.
The rest of the above analysis applies verbatim, and we conclude
\cref{lem:iid_reduction} by combining \eqref{eq:stretchedexpBi} with a union bound over $i\in\lb0,N\rb$.
\end{proof}


\subsection{Proof of \cref{thm:comparison}: zero temperature}\label{sec:zerocomp}

We denote by $\G$ the left-most geodesic $(0,0)\to(0,n)$, and by $\g_i$ the left-most geodesic $(J^{1/2},s_i)\to(J^{1/2},t_i)$, for $i\in\lb1,N\rb$.

By the correspondence of \cref{sec:pos-zero}, the proofs of \cref{lem:cons_reduction,lem:hwy_reduction} also imply the analogous zero temperature statements:
\begin{lemma}[The geodesic is constrained]
    There exists an event $\EE_1$ with $\P(\EE_1)=1-o(1)$, such that $\G\in\Pi^\cons$ on $\EE_1$.
\end{lemma}

\begin{lemma}[Every geodesic $\g_i$ is a local highway]
    There exists an event $\EE_2$ with $\P(\EE_2)=1-o(1)$, such that on $\EE_2$ it holds that $\g_i\in\Pi_i^\hwy$ for all $i\in\lb 1,N\rb$.
\end{lemma}

We define
\[
    x_0^*\coloneqq 0,\quad x_{N+1}^*\coloneqq 0,\quad\text{and}\quad x_i^*\coloneqq \g_i(m_i) \text{ for } i\in\lb 1,N\rb,
\]
as well as
\[
    L_i^*\coloneqq L(x_i^*,m_i;x_{i+1}^*,m_{i+1})-\o(x_{i+1}^*,m_{i+1})\quad\text{for } i\in\lb 0,N-1\rb
\]
and $L_N^* \coloneqq  L(x_N^*,m_N; x_{N+1}^*,m_{N+1})$.
We also write
\[
    L^*\coloneqq \sum_{i=0}^{N} L_i^*,\quad L^0\coloneqq \sum_{i=0}^{N}L_i^0, \quad\text{and}\quad L\coloneqq L(0,0;0,n).
\]
The zero temperature analogues of \cref{lem:argmax_reduction,lem:iid_reduction} are as follows.
\begin{lemma}\label{lem:argmax_reduction0}
    On $\EE_1\cap\EE_2$, we have that $L=L^*$.
\end{lemma}
\begin{lemma}\label{lem:iid_reduction0}
    As $n\to\infty$,
    \[
        \frac{1}{\sqrt n}\left|L^*-L^0\right|\pto 0.
    \]
\end{lemma}
As with the positive temperature case, \cref{lem:argmax_reduction0,lem:iid_reduction0} together imply the approximation \eqref{eq:zero_temp_comparison}, since $\P(\EE_1\cap\EE_2)=1-o(1)$.
Moreover, as \cref{lem:coalescence} trivializes the coupling constructions of \cref{sec:poscomp}, the arguments can be substantially shortened.
\begin{proof}[Proof of \cref{lem:argmax_reduction0}]
    Fix $\o\in\EE_1\cap\EE_2$.
    By planarity,
    \[
        \g_i\cap\G\cap \H_{\lb s_i,\,\,s_i+J^{3/4}\rb} \ne \varnothing
        \quad\text{and}\quad
        \g_i\cap\G\cap \H_{\lb t_i-J^{3/4},\,\,t_i\rb} \ne\varnothing
        \quad\text{for all } i\in\lb 1,N\rb.
    \]
    Therefore by \cref{lem:coalescence}, $\G(m_i)=x_i^*$ for all $i\in\lb 0, N+1\rb$.
    Another application of \cref{lem:coalescence} implies that $L=L^*$.
\end{proof}

For $i\in\lb 0,N\rb$ we denote by $\G_i^0$ the geodesic $(0,m_i)\to(0,m_{i+1})$.
\begin{proof}[Proof of \cref{lem:iid_reduction0}]
    Fix $\o\in\EE_1\cap\EE_2$.
    We have $L=L^*$ by \cref{lem:argmax_reduction0}.
    By planarity,
    \[
        \G_i^0\cap\G\cap \H_{\lb s_i,\,\,s_i+J^{3/4}\rb} \ne \varnothing
        \quad\text{and}\quad
        \G_i^0\cap\G\cap \H_{\lb t_i-J^{3/4},\,\,t_i\rb} \ne\varnothing
        \quad\text{for all } i\in\lb 1,\,N-1\rb,
    \]
    and therefore by \cref{lem:coalescence},
    \[
        \G_i^0\cap\H_{\lb s_i+J^{3/4},\,\, t_i-J^{3/4}\rb} = \G\cap\H_{\lb s_i+J^{3/4},\,\, t_i-J^{3/4}\rb}
        \quad\text{for all } i\in\lb 1,\,N-1\rb.
    \]
    Similarly,
    \[
        \G_0^0\cap\H_{\lb 0,\, s_1\rb}=\G\cap\H_{\lb 0,\, s_1\rb} \quad\text{and}\quad \G_N^0\cap\H_{\lb t_N,\, n\rb}=\G\cap\H_{\lb t_N,\, n\rb}.
    \] 
    It follows that the symmetric difference $(\G\setminus\G^0)\cup(\G^0\setminus\G)$ is a subset of $\BB$, where we define $\G^0\coloneqq \bigcup_i\G^0_i$.
    This implies
    \[
        |L-L^0| \le C N(\log n)^{C'}\,\sup_{v\in\BB}\o(v).
    \]
    \cref{lem:iid_reduction0} now follows from \eqref{eq:unionbound}.
\end{proof}

\section{Fluctuations for free energy with endpoint near the vertical}\label{sec:near_V}

We make a detour to indicate how the ideas of \cref{lem:iid_reduction,lem:iid_reduction0} can be adapted to prove \cref{cor:near_V}.
As discussed in \cref{sec:main_results}, given \cref{thm:main}, the substance of \cref{cor:near_V} is the approximation 
\begin{equation}\label{eq:cor_approx}
    \frac{1}{\sqrt n}|G(0,0;0,n) - G(0,0;y_n,n)| \pto  0\qquad\text{whenever}\quad y_n=o(\sqrt n),
\end{equation}
which we now establish.
\begin{proof}[Proof of \eqref{eq:cor_approx}]
Assume first $G=F$.
Consider polymers $\pi:(0,0)\to(0,n)$ and $\pi^y : (0,0)\to(y_n,n)$, with polymer measures denoted  respectively by $\Q,\Q^y$.
Fix any sequence $w_n$ satisfying $w_n-y_n\uparrow \infty$ and $w_n=o(\sqrt n)$.
By \cref{thm:pinning}, for all sufficiently large $n$ there exists an event $\B$ with $\P(\B)\ge 0.9$ and
\begin{equation}\label{eq:Qy}
    \Q^y\left(\pi^y\cap\V_{\lb n-w_n,\,\,n-y_n\rb} \ne \varnothing\right) \ge 0.99\quad\text{on } \B.
\end{equation}
Assume the event inside $\Q^y$ above occurs.
Then by planarity, $\pi$ and $\pi^y$ intersect inside the box $\lb 0,y_n\rb \times \lb n-y_n, n\rb$.
The proof of \cref{lem:iid_reduction} (viz. \eqref{eq:pigeon3}) now implies that for some $a,a^y\in\lb n-w_n,\,n-y_n\rb$,
\[
    |F(0,0;0,n) - F(0,0;y_n,n)|
    \ls F(1,a+1;0,n) + F(1,a^y+1;w_n,n) + \log (w_n-y_n).
\]
The arguments of \eqref{eq:crude_upper} and \eqref{eq:unionbound} imply that the right side above is $o(\sqrt n)$ with $\P$-probability $1-o(1)$.
The case $G=L$ follows by modifying the proof of \cref{lem:iid_reduction0} in an analogous manner.
\end{proof}

We now establish the linear growth $\Var(G(0,0;y_n,n))\asymp n$.
The proof of \cref{lem:varatmostlinear} applies verbatim to show $\Var(G(0,0;y_n,n))\ls n$.
We prove the lower bound for $G=F$ via a slight modification of the proof of \cref{lem:varatleastlinear} (the same approach works for $G=L$).
Fix $j\in\lb 0,n-w_n\rb$.
Recall that in \eqref{eq:Aj} we defined
\begin{equation}\label{eq:Aj2}
    \A_j\coloneqq \bigl\{\o:\Q\bigl(\pi(j)\in\lb0,x_0\rb\bigr)\ge 0.1\bigr\}.
\end{equation}
Fix $\o\in\A^y_j\coloneqq \A_j\cap\B$ (the event $\B$ is defined above \eqref{eq:Qy}).
Let $\QQ$ be the coupling of $\Q,\Q^y$ provided by \cref{lem:polymer_ordering}.
By planarity and \eqref{eq:Qy}, the polymers $\pi$ and $\pi^y$ typically coincide on $\H_{\lb 0,\,n-w_n\rb}$:
\[
    \QQ\bigl(\pi(i) = \pi^y(i)\bigr) \ge 0.99\quad\text{for all } i\in\lb 0,\,n-w_n\rb.
\]
Therefore by \eqref{eq:Aj2},
\[
    \Q^y\bigl(\pi^y(j) \in \lb0,x_0\rb\bigr) \ge 0.09.
\]
As $\P(\A_j^y)\ge 0.8$ and $n-w_n\asymp n$, the claimed lower bound $\Var(F(0,0;y_n,n))\gs n$ now follows from the proof of \cref{lem:varatleastlinear}. \hfill $\square$


\section{Lindeberg condition}\label{sec:lindeberg}

In this section we prove \cref{thm:main} by combining the preceding results with the Lindeberg central limit theorem.
We recall the latter (e.g. \cite[Theorem 27.2]{BilProbabilityMeasure1995}):
\begin{theorem}[Lindeberg central limit theorem]\label{thm:lindeberg_clt}
    Let $\{\xi_{N,i} : N\ge 1, \, i\in\lb 0,N\rb\}$ be a \emph{triangular array}, i.e. a collection of random variables such that for any $N$, the random variables $\xi_{N,0}, \dots,\xi_{N,N}$ are independent.
    Suppose that $\bE[\xi_{N,i}]=0$ for all $N,i$ and that $\sum_{i=0}^N\Var(\xi_{N,i})=1$ for all $N$.
    Suppose also that
    \begin{equation}\label{eq:lindeberg_condition}
        \lim_{N\to\infty}\sum_{i=0}^{N}\bE\bigl[\xi_{N,i}^2\,\1_{|\xi_{N,i}|>\e}\bigr] = 0\quad\text{for all }\e>0.
    \end{equation}
    Then as $N\to\infty$ we have the convergence in distribution
    \[
        \sum_{i=0}^{N} \xi_{N,i} \dto  \msf{N}(0,1).
    \]
\end{theorem}

Fix $G\in\{F,L\}$.
Let all notation be as in \eqref{eq:JK}, \eqref{eq:si_mi_ti}, \eqref{eq:def_Fi}, so that $K\asymp n^{0.9}$ and $N\asymp n/K \asymp n^{0.1}$ and $m_{i+1}-m_i\asymp K$ for $i\in\lb 0,N-1\rb$ (and $m_{N+1}-m_N\ls K$).
By the variance estimates of \cref{lem:varatleastlinear,lem:varatmostlinear}, there exists $\sigma_N\asymp 1$ such that the triangular array
\[
    X_{N,i}\coloneqq \frac{G(0,m_i;\,0,m_{i+1})-\bE[G(0,m_i;\,0,m_{i+1})]}{\sigma_N \sqrt{NK}},\qquad N\ge 1,\,\,i\in\lb 0,N\rb
\]
satisfies $\sum_{i=0}^{N} \Var(X_{N,i})=1$ for all $N$.
We also have $\bE[X_{N,i}]=0$ for all $N,i$.
Finally, we claim that the $X_{N,i}$ also satisfy the Lindeberg condition \eqref{eq:lindeberg_condition}:
\begin{equation}\label{eq:lindeberg}
    \lim_{N\to\infty}\sum_{i=0}^N\E\bigl[X_{N,i}^2\,\1_{|X_{N,i}|>\e}\bigr] = 0 \quad\text{for all }\e>0.
\end{equation}
Given \eqref{eq:lindeberg}, we can combine \cref{thm:lindeberg_clt} with \cref{thm:comparison} to deduce \cref{thm:main}.

\begin{proof}[Proof of \eqref{eq:lindeberg}]
We make a straightforward modification of the proof of the large deviations estimate \cref{lem:large_dev}.
Fix $i\in\lb 0,N\rb$ and write $X\coloneqq X_{N,i}$.
Let $\wh X$ denote the same, but with respect to the truncated environment $\wh\o(x,t)\coloneqq \o(x,t)\wedge n^{0.04}$.
Let $\FF_j$ be the $\sigma$-algebra generated by $\wh\o$ up to height $j$, i.e, $\FF_j\coloneqq\sigma\bigl(\wh\o(x,t):x\in\Zpos,\;t\le j\bigr)$.
The argument of \eqref{eq:conc_bounded_inc} implies that
\[
    \left|\E\bigl[\wh X\given[\FF_j]\bigr] - \E\bigl[\wh X\given[\FF_{j-1}]\bigr]\right| \le 2\frac{n^{0.04}}{\sigma_N\sqrt{NK}}\le Cn^{-0.46}\quad\text{for all } j\in\lb m_i,\,m_{i+1}\rb,
\]
where we used that $NK\asymp n$.
It follows from the Azuma--Hoeffding inequality that, for any $z>0$,
\[
     \P\Bigl(\bigl|\wh X\bigr| > z\Bigr) \le 
     C\exp\left(-C'\frac{z^2}{(m_{i+1}-m_i)\cdot (n^{-0.46})^2}\right)
     \le C\exp\left(-C'z^2\, n^{0.02}\right),
\]
where we used that $m_{i+1}-m_i\asymp K\asymp n^{0.9}$.
Also, the arguments of \eqref{eq:conc_union_bound}, \eqref{eq:conc_expectation} apply verbatim to show that
\[
    0 \le \E\bigl[X - \wh X\bigr] \le C\exp\left(-C'n^{0.02}\right).
\]

Fix $\e>0$.
By combining the above two displays with the argument of \eqref{eq:conc_final}, we conclude that there exists $n_0>0$ such that for all $n\ge n_0$ and for all $z\ge \e$,
\begin{equation}\label{eq:lindeberg_tail}
    \P\bigl(|X|>z\bigr) \le C\exp\left(-C'\,\min\{1,z^2\}\,n^{0.02}\right).
\end{equation}
It follows that $\bE[X^4]=O(1)$.
By the Cauchy--Schwarz inequality and another application of \eqref{eq:lindeberg_tail},
    \[
        \bE[X^2\,\1_{|X|>\e}] 
        \le \bE[X^4]^{1/2}\cdot \P\bigl(|X|>\e\bigr)^{1/2} \le C\exp\left(-C'\,\min\{1,\e^2\}\,n^{0.02}\right).
    \]
Finally, since $n^{0.02}\asymp N^{0.2}$, the above display implies the Lindeberg condition \eqref{eq:lindeberg}.
\end{proof}

\emergencystretch=1em
\printbibliography

\end{document}